\documentclass[12pt,a4paper]{article}
\usepackage{amsmath,amssymb,amsthm,graphicx,color,bbm,array}
\usepackage[left=1in,right=1in,top=1in,bottom=1in]{geometry}
\usepackage{cite}
\usepackage[british]{babel}
\usepackage{hyperref} 
\usepackage{enumitem}
\usepackage{tikz}
\usepackage{caption}
\usepackage{subcaption}
\hypersetup{
    colorlinks=false,
    pdfborder={0 0 0},
}

\newtheorem{theorem}{Theorem}[section]
\newtheorem{corollary}[theorem]{Corollary}
\newtheorem{proposition}[theorem]{Proposition}
\newtheorem{lemma}[theorem]{Lemma}

\numberwithin{equation}{section}

\theoremstyle{definition}
\newtheorem{definition}[theorem]{Definition}

\theoremstyle{remark}
\newtheorem{remark}[theorem]{Remark}
\newtheorem{remarks}[theorem]{Remarks}
\newtheorem*{remark*}{Remark}

 \allowdisplaybreaks

\newcommand{\1}[1]{{\mathbbm{1}\mkern -1.5mu}{\{#1\}}}

\newcommand{\R}{{\mathbb R}}
\newcommand{\Z}{{\mathbb Z}}
\newcommand{\N}{{\mathbb N}}
\newcommand{\barN}{{\overline \N}}
\newcommand{\ZP}{{\mathbb Z}_+}
\newcommand{\RP}{{\mathbb R}_+}

\DeclareMathOperator{\Exp}{\mathbb{E}}
\let\Pr\relax
\DeclareMathOperator{\Pr}{\mathbb{P}}

\DeclareMathOperator{\bE}{{\mathbf{E}}}
\DeclareMathOperator{\bP}{{\mathbf{P}}}
\DeclareMathOperator{\bVar}{{\mathbf{V}\!ar}}

\DeclareMathOperator{\Var}{\mathbb{V}ar}

\newcommand{\Prf}{{\Pr}^{N,M}_z}
\newcommand{\Expf}{{\Exp}^{N,M}_z}
\newcommand{\Prfo}{{\Pr}^{N,M}_{(1,M)}}

\newcommand{\PrfM}{{\Pr}^{N_M,M}_{z_M}}
\newcommand{\PrfMo}{{\Pr}^{N_M,M}_{(1,M)}}
\newcommand{\ExpfM}{{\Exp}^{N_M,M}_{z_M}}
\newcommand{\ExpfMo}{{\Exp}^{N_M,M}_{(1,M)}}
\newcommand{\VarfMo}{{\Var}^{N_M,M}_{(1,M)}}
\newcommand{\bPo}{\bP_1}
\newcommand{\bEo}{\bE_1}

\newcommand{\tbmo}{\tau^{\mbox{\textup{\tiny BM}}}_1}
\newcommand{\tbmoo}{\tau^{\mbox{\textup{\tiny BM}}}_{0,1}}

\newcommand{\phidm}{\varphi_{\mathrm{DM}}}
\newcommand{\DM}{\mathrm{DM}(1/2)}
\newcommand{\Stable}{\mathrm{S}_+(1/2)}

\newcommand{\eps}{\varepsilon}

\newcommand{\re}{{\mathrm{e}}}

\newcommand{\ud}{{\mathrm d}}

\newcommand{\cC}{{\mathcal C}}

\newcommand{\cE}{{\mathcal E}}
\newcommand{\cF}{{\mathcal F}}

\newcommand{\cI}{{\mathcal I}}

\newcommand{\cS}{{\mathcal S}}
\newcommand{\cT}{{\mathcal T}}

\newcommand{\Mgf}{\psi_M}

\newcommand{\cExp}{{\cE_1}}
\newcommand{\bPhi}{{\overline \Phi}}

\newcommand{\as}{\ \text{a.s.}}

\newcommand{\tod}{\overset{\mathrm{d}}{\longrightarrow}}

\newcommand{\biggmid}{\; \biggl| \;}

\newcommand{\eqd}{\overset{d}{=}}

\newcommand{\bI}{\partial I}
\newcommand{\iI}{I^{{\raisebox{0.3pt}{\scalebox{0.8}{$\circ$}}}}}

\makeatletter
\def\namedlabel#1#2{\begingroup  
    (#2)%
    \def\@currentlabel{#2}%
    \phantomsection\label{#1}\endgroup
}
\makeatother

\newlist{myenumi}{enumerate}{10}
\setlist[myenumi]{leftmargin=0pt, labelindent=\parindent, listparindent=\parindent, labelwidth=0pt, itemindent=!, itemsep=1pt, parsep=4pt}

\newlist{thmenumi}{enumerate}{10}
\setlist[thmenumi]{leftmargin=0pt, labelindent=\parindent, listparindent=\parindent, labelwidth=0pt, itemindent=!}

\title{Energy-constrained random walk\\ with boundary replenishment}
\author{Andrew Wade\footnote{\scalebox{0.94}{Department of Mathematical Sciences, Durham University, Durham; \href{mailto:andrew.wade@durham.ac.uk}{\texttt{andrew.wade@durham.ac.uk}}.}}
 \and Michael Grinfeld\footnote{\scalebox{0.94}{Department of Mathematics \& Statistics, University of Strathclyde, Glasgow; \href{mailto:m.grinfeld@strath.ac.uk}{\texttt{m.grinfeld@strath.ac.uk}}.}}}

\date{1 September 2023}

\begin{document}
\maketitle

\begin{abstract}
We study an energy-constrained random walker on a length-$N$ interval of the one-dimensional integer lattice, with boundary reflection.
The walker consumes one unit of energy for every step taken in the interior, and energy is replenished up to a capacity of~$M$
on each boundary visit. 
We establish large $N, M$ distributional asymptotics for the lifetime of the walker, 
i.e., the first time at which the walker runs out of energy while in the interior. 
Three phases are exhibited. 
When $M \ll N^2$ (energy is scarce),
we show that there is an $M$-scale limit distribution related to a Darling--Mandelbrot law,
while when $M \gg N^2$ (energy is plentiful) we show that there is an exponential limit distribution on a stretched-exponential scale.
In the critical case where $M / N^2 \to \rho \in (0,\infty)$, we show that there is an $M$-scale limit
in terms of an infinitely-divisible distribution expressed via certain theta functions.
\end{abstract}

\medskip


\medskip


\section{Introduction}

This paper is motivated by a number of problems in mathematical ecology.
The motion of individual animals and the space-time statistics of
animal populations are of central interest in ecology, crucial to the
understanding of population structure and dynamics, dispersion
patterns, foraging, herding, territoriality, and other aspects of the
behaviour of animals and their interactions with and responses to
their environment and broader ecosystem~\cite{gb,hjmm,pl,sby,vlrs}.
Movement ultimately bears on large-scale (in space and time)
phenomena, such as biological fitness, 
genetic variability, and inter-species dynamics.

Random walks, diffusions, and related processes have been used for over a century
to model animal movement: see e.g.~\cite{codling,jpe,vlrs} for an
introduction to the extensive literature on these topics. 
Assuming Fickian diffusion leads to
 reaction--diffusion models in which dynamics is driven by smooth
gradients of resource or habitat, such 
as chemotactic or haptotactic dynamics in microbiology~\cite{tw}, 
but fails in several important respects
to reflect the real-world behaviour of animals~\cite{gphm,jpe}. 
Various modelling paradigms attempt to incorporate more realistic aspects of animal behaviour, such
as memory and persistence~\cite{gurarie}, intermittent rest
periods~\cite{tilles}, or anomalous diffusion~\cite{jpe}.

A basic aspect of population dynamics is the flow of energy:
individuals consume resource and subsequently expend energy in somatic
growth, maintenance, reproduction, foraging, and so on; penguins must
balance feeding and swimming, for example~\cite{brfm}.  
The
distribution of scarce food or water in the environment imposes
constraints on animal movement, as is seen, for instance, in flights
of butterflies between flowers, elk movements between feeding craters,
and elephants moving between water sources in the dry
season~\cite{gurarie,wato}. 
An important modelling challenge
is to incorporate resource heterogeneity in space and time, to model the
distribution of food and shelter, consumption of resources, and
seasonality; in bounded domains, one must provide a well-motivated choice of boundary conditions.

Traditionally, mathematical ecology modelling of spatially distributed
processes in bounded domains employs Dirichlet or Neumann boundary conditions, which do
not do justice to the wealth of behaviour that can be exhibited at the
boundary of a domain occupied by an animal population. Clearly,
availability of resources can be different in the interior of the
domain and at the boundary, as can the nature of interactions among
members of the population. It is not impossible that an animal may
want to spend some time staying at, or diffusing along the boundary. An
example of the differences in behaviours, this time of a population of
molecules, is provided in the pregnancy-test model of~\cite{mgm}, where
the species only diffuse in the bulk of the sample, but participate in
reactions with the coating of the vessel where the test takes place. Another area in which reaction--diffusion
systems are applied and in which boundary conditions enter in a crucial way is porous-medium transport~\cite{cr,td}.

There have been proposed random walk models where arrival on the
boundary results in the termination of the walk; see~\cite{bs} for a review. 
In
the present paper we  take a ``dual'' view, in which staying in
the interior of the domain leads ultimately to the demise of the
walker, while arrival at the boundary  provides \emph{viaticum} to replenish the walker's energy and allows the walk to continue.
For example, an island in a lake may support animals that roam the interior but must return to shore to drink. Adjacent models 
include
the ``mortal random walks'' of~\cite{baak}
and the ``starving random walks'' of~\cite{ykr}, but in neither of these does the walker carry an internal energy state.
Our model differs from models that incorporate
resource depletion by feeding~\cite{bl,bbkr,bcr,cbr,grebenkov} in that, for us, energy replenishment only occurs on the boundary
and the resource is inexhaustible.
Comparison of our
results to those for models including resource depletion, in which the
domain in which the walk is in danger of extinction grows over time, is a topic we hope to address in future work.

Section~\ref{sec:results} describes our mathematical set-up and our main results, starting, in Section~\ref{sec:general-model}
with an introduction of a class of Markov models of energy-constrained random walks with boundary replenishment.

\section{Model and main results}
\label{sec:results}

\subsection{Energy-constrained random walk}
\label{sec:general-model}

For $N \in \N := \{1,2,3,\ldots\}$, denote the finite discrete interval $I_N := \{ x \in \Z : 0 \leq x \leq N \}$.
Also define the semi-infinite interval by $I_\infty := \ZP := \{ x \in \Z : x \geq 0 \}$.
We write $\barN := \N \cup \{ \infty \}$ and then $I_N, N \in \barN$ includes both finite and infinite cases.
The boundary $\bI_N$ of $I_N$ is defined as $\bI_N := \{0,N\}$ for $N \in \N$,
and $\bI_\infty := \{0\}$ for $N = \infty$; the interior is $\iI_N := I_N \setminus \bI_N$.
We suppose $N \geq 2$, so that $\iI_N$ is non-empty. Over $\barN$ we define simple arithmetic and function evaluations in the way consistent
with taking limits over $\N$,  e.g., $1/\infty := 0$, $\exp \{ -\infty \} := 0$, and so on.
 
We define a class of discrete-time Markov chains
 $\zeta := (\zeta_0, \zeta_1, \ldots)$, where $\zeta_n := (X_n,\eta_n) \in I_N \times \ZP$,
whose transition law is determined by an \emph{energy update} (stochastic) matrix $P$ over $\ZP$, i.e., a function $P : \ZP^2 \to [0,1]$
with $\sum_{j \in \ZP} P(i,j) = 1$ for all $i \in \ZP$.
The coordinate $X_n$
represents the location of a random walker, and $\eta_n$ its current energy level. 
Informally, the dynamics of the process are as follows. 
As long as it has positive
energy and is in the interior, the walker performs simple random walk steps; each step uses one unit of energy.
If the walker runs out of energy while in the interior, the walk terminates.
If the process is at the boundary $\bI_N$ with current energy level $i$, $P(i,j)$ is the probability that the energy changes to level~$j$; the walk reflects into the interior.

Formally, the transition law is as follows.
\begin{itemize}
\item \emph{Energy-consuming random walk in the interior:} If $i \in \N$ and $x \in \iI_N$, then
\begin{equation}
\label{eq:transition-1}
\Pr ( X_{n+1} = X_n + e, \, \eta_{n+1} = \eta_n - 1 \mid X_n = x , \eta_n = i ) = \frac{1}{2}, \text{ for } e \in \{-1,+1\} .\end{equation}
\item \emph{Extinction through exhaustion:} If $x  \in \iI_N$, then
\begin{equation}
\label{eq:transition-2}
 \Pr ( X_{n+1} = x, \, \eta_{n+1} = 0 \mid X_n = x, \eta_n = 0 ) = 1 .\end{equation}
\item \emph{Boundary reflection and energy transition:} If $i \in \ZP$, $x \in \bI_N$, and $y \in \iI_N$ is the unique~$y$
such that $|y-x| = 1$, then
\begin{equation}
\label{eq:transition-3}
 \Pr ( X_{n+1} = y , \, \eta_{n+1} = j \mid X_n = x, \eta_n =i )  = P ( i, j) .
\end{equation}
\end{itemize}

In the present
paper, we focus on a model with \emph{finite energy capacity} and maximal energy replenishment at the boundary. 
This corresponds to a
specific choice of $P$, namely $P(i, M)=1$, as described in the next
section. Other choices of the update matrix $P$ are left for future work.

\subsection{Finite energy capacity}
 \label{sec:finite-capacity}

Our finite-capacity model has a parameter $M \in \N$, representing
the maximum energy capacity of the walker. The boundary energy update rule that we
take is that energy is always replenished up to the maximal level~$M$. Here is the definition.

\begin{definition}
\label{def:finite-capacity}
For $N \in \barN, M \in \N$, and $z \in I_N \times I_M$, 
the finite-capacity $(N,M,z)$-model is the Markov chain~$\zeta$
with initial state $\zeta_0 = z$ and with transition law 
defined through~\eqref{eq:transition-1}--\eqref{eq:transition-3}
with energy update matrix $P$ given by~$P(i,M) =1$ for all $i \in \ZP$.
\end{definition}

For the $(N,M,z)$ model from Definition~\ref{def:finite-capacity}, with $z = (x,y) \in I_N \times I_M$, we write $\Prf$ and $\Expf$ for probability and expectation
under the law of the corresponding Markov chain $\zeta$ with spatial domain~$I_N$, energy capacity~$M$, initial location~$X_0=x \in I_N$,
and initial energy $\eta_0=y \in I_M$.
The main quantity of interest for us here is 
the \emph{total lifetime} (i.e., time before extinction) of the process, defined by
\begin{equation}
\label{eq:def-lifetime}
\lambda := \min \{ n  \in \ZP: X_n \in \iI_N , \, \eta_n = 0 \},
\end{equation} 
where we adopt the usual convention that $\min \emptyset:=+\infty$.
The  process $\zeta \in I_N \times I_M$
can be viewed as a two-dimensional random walk with
a reflecting/absorbing boundary, 
 in which the interior drift (negative in the energy component)
competes against the (positive in energy) boundary reflection: see Figure~\ref{fig:strip} for a schematic.

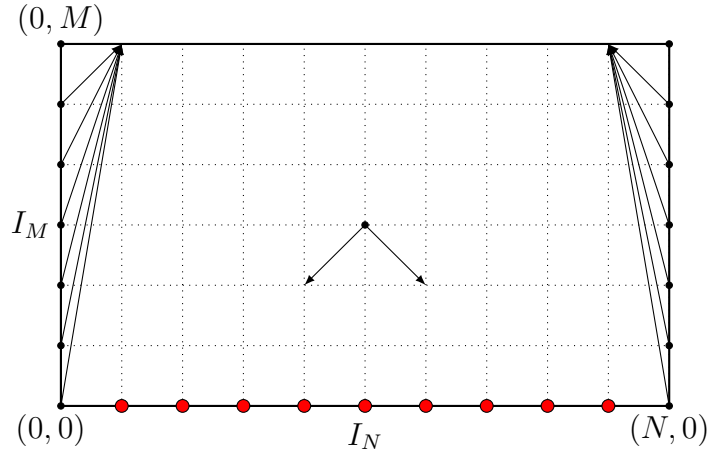
\begin{figure}[!ht]
\begin{center}
\begin{tikzpicture}[domain=0:8, scale = 0.8]
\filldraw (0,0) circle (1.5pt);
\filldraw (10,0) circle (1.5pt);
\node at (10,-0.4) {$(N,0)$};
\node at (0,6.4) {$(0,M)$};
\node at (0,-0.4) {$(0,0)~~$};
\draw[step=1.0,black,dotted,thin,xshift=0.0cm,yshift=0.0cm] (0,0) grid (10,6);
\node at (5, -0.5)       {$I_N$};
\node at (-0.5, 3)       {$I_M$};
\draw[black,thick] (0,0) -- (10,0);
\draw[black,thick] (0,0) -- (0,6);
\draw[black,thick] (10,0) -- (10,6);
\draw[black,thick] (0,6) -- (10,6);
\filldraw (5,3) circle (1.5pt);
\draw[black,->,>=latex] (5,3) -- (6,2);
\draw[black,->,>=latex] (5,3) -- (4,2);
\filldraw[fill=red] (1,0) circle (3pt);
\filldraw[fill=red] (2,0) circle (3pt);
\filldraw[fill=red] (3,0) circle (3pt);
\filldraw[fill=red] (4,0) circle (3pt);
\filldraw[fill=red] (5,0) circle (3pt);
\filldraw[fill=red] (6,0) circle (3pt);
\filldraw[fill=red] (7,0) circle (3pt);
\filldraw[fill=red] (8,0) circle (3pt);
\filldraw[fill=red] (9,0) circle (3pt);
\filldraw (0,1) circle (1.5pt);
\filldraw (0,2) circle (1.5pt);
\filldraw (0,3) circle (1.5pt);
\filldraw (0,4) circle (1.5pt);
\filldraw (0,5) circle (1.5pt);
\filldraw (10,1) circle (1.5pt);
\filldraw (10,2) circle (1.5pt);
\filldraw (10,3) circle (1.5pt);
\filldraw (10,4) circle (1.5pt);
\filldraw (10,5) circle (1.5pt);
\filldraw (10,6) circle (1.5pt);
\filldraw (0,6) circle (1.5pt);
\draw[black,->,>=latex] (0,0) -- (1,6);
\draw[black,->,>=latex] (0,1) -- (1,6);
\draw[black,->,>=latex] (0,2) -- (1,6);
\draw[black,->,>=latex] (0,3) -- (1,6);
\draw[black,->,>=latex] (0,4) -- (1,6);
\draw[black,->,>=latex] (0,5) -- (1,6);
\draw[black,->,>=latex] (10,0) -- (9,6);
\draw[black,->,>=latex] (10,1) -- (9,6);
\draw[black,->,>=latex] (10,2) -- (9,6);
\draw[black,->,>=latex] (10,3) -- (9,6);
\draw[black,->,>=latex] (10,4) -- (9,6);
\draw[black,->,>=latex] (10,5) -- (9,6);
\end{tikzpicture}
\end{center}
\vspace{-4mm}
\caption{\label{fig:strip} Schematic of the transitions for the reflecting random walk $\zeta_n = (X_n, \eta_n)$ in the rectangle $I_N \times I_M$. 
The arrows indicate some possible transitions: in the interior, energy decreases and the particle executes a nearest-neighbour random walk, while the boundary states provide energy replenishment. 
The larger (red-filled) circles indicate the absorbing states (zero energy, but not at the boundary).}
\end{figure}

If for initial state $z = (x,y) \in I_N \times I_M$ it holds that $N > M+x$ (including $N = \infty$), then 
 the energy constraint and the fact that $X_0 =x$ ensures that $X_n$ can never exceed $M+x$, 
so $X_n$ is constrained to the finite interval $I_{M+1+x}$. In other words, every $(N,M,z)$ model with $N > M+x$ is equivalent to the $(\infty,M,z)$ model, for any $y \in I_M$.
Since we start with $\eta_0 = y$ and energy decreases by at most one unit per unit time, $\Prf ( \lambda \geq y ) = 1$.
The following `irreducibility' result,
proved in Section~\ref{sec:proofs:general},
 shows that extinction is certain, i.e., $\Prf (\lambda < \infty ) =1$, provided there are at least two sites in $\iI_N$.

\begin{lemma}
\label{lem:irreducibility}
Suppose that $N \in \barN$ with $N \geq 3$, that $M \in \N$, and that $\zeta_0 = z = (x,y)$.
Under~$\Prf$, the process~$\zeta$ is a time-homogeneous Markov chain on the finite state space
 $\Lambda_{N,M} := I_{N \wedge (M+x)} \times I_M$.
 Moreover, there exists $\delta >0$ (depending only on $M$)
such that $\sup_{z \in I_N \times I_M} \Expf [ \re^{\delta \lambda} ] < \infty$.
\end{lemma}

Our main results concern distributional asymptotics for the random variable $\lambda$
as $N, M \to \infty$. We will demonstrate a phase transition depending on the relative growth of $M$ and $N$. Roughly speaking, 
since in time $M$ a simple random walk typically travels distance of order $\sqrt{M}$, it turns out that if $M \gg N^2$ the random walk will visit the
boundary many times, and $\lambda$ grows exponentially in $M/N^2$, 
while if $M \ll N^2$ then there are relatively few (in fact, of order~$\sqrt{M}$)
visits to the boundary before extinction, and $\lambda$ is of order~$M$. The case $M \ll N^2$, where the capacity constraint dominates,
we call the \emph{meagre-capacity} limit, and we treat this case first;
in the limit law appears the relatively unusual \emph{Darling--Mandelbrot} distribution (see Section~\ref{sec:meagre-capacity}). The case $M \gg N^2$, where energy is plentiful,
we call the \emph{confined-space} limit, and (see Section~\ref{sec:confined-space}) the limit law is exponential, as might be expected due to the small rate of extinction in that case.
The critical case, where $M /N^2 \to \rho \in (0,\infty)$ is dealt with in Section~\ref{sec:critical}.
Section~\ref{sec:heuristics} gives an outline, at the level of heuristics, of the proofs.

\subsection{The meagre-capacity limit}
\label{sec:meagre-capacity}

Our first result looks at the case where $N, M$ are both large but $M \in \N$ is small (in a sense we quantify) with respect to $N \in \barN$.
This will include all the $(\infty,M,z)$ models, which, as remarked above,
coincide with the $(N,M,z)$ models for $N > M+x$. 
However, it turns out that the $(\infty,M,z)$ behaviour is also asymptotically replicated in  $(N,M,z)$ models
for $N \gg \sqrt{M}$. The formal statement is Theorem~\ref{thm:meagre-capacity} below. 

To describe the limit 
distribution in the theorem, we define, for $t \in \R$, 
\begin{equation}
\label{eq:E-def}
 \cI (t) := t \int_0^1  u^{-1/2} \re^{u t} \ud u , \text{ and } t_0 := \inf \{ t \in \R : \re^t - \cI(t) \leq 0 \}.
\end{equation}
Then $t_0 \approx 0.8540326566$ (see Lemma~\ref{lem:kummer} below), and  
\begin{equation}
\label{eq:xi-transform}
\phidm (t) := \frac{1}{\re^{t} - \cI (t)}, \text{ for } t < t_0,
\end{equation}
defines the moment generating function of a distribution on~$\RP$
known as the \emph{Darling--Mandelbrot distribution} with parameter $1/2$. We write
$\xi \sim \DM$ to mean that $\xi$ is an $\RP$-valued random variable
with $\Exp [ \re^{t \xi} ] = \phidm(t)$, $t < t_0$. We refer to Section~\ref{sec:heuristics} for
an heuristic explanation
behind the appearance in our Theorem~\ref{thm:meagre-capacity} below of the $\DM$ distribution,
based on its role in the theory of heavy-tailed random sums.
We also refer
to 
Lemma~\ref{lem:xi-moments} for the moments of $\DM$, the first of which is~$\Exp \xi = 1$.

The result in this section is stated for a sequence of  $(N_M,M,z_M)$ models, indexed by $M \in \N$, with
 $z_M = (x_M,y_M) \in I_{N_M} \times I_M$ and $N_M \in \barN$ 
satisfying the \emph{meagre-capacity} assumption that, for some $a \in [0,\infty]$ and $u \in (0,1]$,
\begin{equation}
\label{eq:meagre-capacity}
\lim_{M \to \infty} \frac{  x^2_M \wedge (N_M - x_M)^2 }{M} = a, ~~\lim_{M \to \infty} \frac{y_M}{M} = u, ~\text{ and } 
\lim_{M \to \infty} \frac{M}{N^2_M} = 0.
\end{equation}
Included in~\eqref{eq:meagre-capacity} is any sequence $N_M$ for which $N_M = \infty$ eventually.

For $b = (b_t, t\in\RP)$ a standard Brownian motion on~$\R$ started from $b_0=0$, define~$\tbmo := \inf \{ t \in \RP : b_t = 1\}$.
Then (see Remarks~\ref{rems:darling}\ref{rems:darling-d} below) $\tbmo$ has a positive $(1/2)$-stable (L\'evy) distribution:
\begin{equation}
\label{eq:tbm}
 \tbmo \sim \Stable, \text{ i.e., $\tbmo$ has probability density $f(t) = (2\pi t^3)^{-1/2} \re^{-1/(2t)}$, $t>0$.} \end{equation}
Write  $\Phi  (s):= \Pr ( b_1 \leq s)$ and $\bPhi (s) := 1 - \Phi(s) = \Pr ( b_1 > s)$, $s \in \R$,
for the standard normal distribution and tail functions. Define
\begin{equation}
\label{eq:g-def}
g (a,u) := u + (4-2u-2a ) \bPhi ( \sqrt{a/u} ) +   \sqrt{\frac{2a u}{\pi} } \re^{-a/(2u)} .
 \end{equation}

Here is the limit theorem in the meagre-capacity case.
The case $N_M=\infty$, $a=0$ 
of~\eqref{eq:lifetime-meagre-limit} 
can be read off from results of~\cite{bs} (see Section~\ref{sec:heuristics}); the other cases, we believe, are new.
We use `$\tod$' to denote convergence in distribution under the implicit probability measure (in the following theorem, namely $\PrfM$). 

\begin{theorem}
\label{thm:meagre-capacity}
Consider the $(N_M,M,z_M)$ model with $M \in \N$, $N_M \in \barN$ and $z_M \in I_{N_M} \times I_M$
such that~\eqref{eq:meagre-capacity} holds.
Then, 
if $\xi \sim \DM$ and $\tbmo \sim \Stable$ are independent with distributions given by~\eqref{eq:xi-transform} and~\eqref{eq:tbm} respectively,
it holds that 
\begin{align}
\label{eq:lifetime-meagre-limit}
 \frac{\lambda}{M} & \tod \min ( u, a \tbmo ) + (1+ \xi) \1 { a \tbmo < u}, \text{ as } M \to \infty; \text{ and} \\
\label{eq:expectation-limit}
\lim_{M \to \infty} \frac{\ExpfM \lambda}{M} & = g (a,u) ,
\end{align}
where $g$ is defined at~\eqref{eq:g-def}.
In particular, if $a=0$, then the limits in~\eqref{eq:lifetime-meagre-limit} and~\eqref{eq:expectation-limit}
are equal to $1+\xi$ and $2 = 1 + \Exp \xi$, respectively, while for $a=\infty$ they are both~$u$.
	\end{theorem}
 
\begin{remarks}
	\phantomsection
\label{rems:darling}
\begin{myenumi}[label=(\alph*)]
\item
\label{rems:darling-a}
 For every $0 < u \leq 1$, the function~$a \mapsto g(a,u)$ on the right-hand side of~\eqref{eq:g-def} 
is strictly decreasing: see Lemma~\ref{lem:g-calculus} below.
\item
If $a >0$, then the distribution of the limit in~\eqref{eq:lifetime-meagre-limit} has an atom at value~$u$
of mass $\Pr ( a \tbmo \geq u ) = 2 \Phi ( \sqrt{a /u} ) -1$, as given by~\eqref{eq:reflection} below; on the other hand, if $a=0$, 
it has a density (since~$\xi$ does, as explained in the next remark). The atom at $u$ represents the possibility that the random walk
runs out of energy before ever reaching the boundary.
\item
\label{rems:darling-b} 
The~$\DM$ distribution specified by~\eqref{eq:xi-transform} appears in a classical result of Darling~\cite{darling}
on maxima and sums of positive $\alpha$-stable random variables in the case $\alpha=1/2$, 
and more recently in the analysis of anticipated rejection algorithms~\cite{bs,lew,louchard}, where it has become
known as the Darling--Mandelbrot distribution with parameter $1/2$. Darling~\cite[p.~103]{darling} works with the characteristic function (Fourier transform); Feller~\cite[p.~465]{feller2} gives the~$t<0$ Laplace transform~\eqref{eq:xi-transform}.
The~$\DM$ distribution has a probability density which is continuous on~$(0,\infty)$,  is non-analytic at integer points, and has no elementary closed form,
but has an infinite series representation, can be efficiently approximated, and its asymptotic properties are known: see~\cite{bs,lew,lh}.
\item
\label{rems:darling-d} To justify~\eqref{eq:tbm}, recall that, by the reflection principle, for $t \in (0,\infty)$, 
\begin{equation}
\label{eq:reflection} \Pr ( \tbmo > t ) = \Pr \biggl( \sup_{0 \leq s \leq t} b_s < 1 \biggr)=  2 \Pr ( b_t < 1 ) -1 = 2 \Phi ( t^{-1/2} ) -1;  \end{equation}
this gives~\eqref{eq:tbm}, cf.~\cite[pp.~173--5]{feller2}.
\end{myenumi}
\end{remarks}

\subsection{The confined-space limit}
\label{sec:confined-space}

We now turn to the second limiting regime. 
The result in this section is stated for a sequence of  $(N_M,M,z_M)$ models, indexed by $M \in \N$, with
 $z_M = (x_M, y_M) \in I_{N_M} \times I_M$ and $N_M \in \N$ (note $N_M < \infty$ now)
satisfying the \emph{confined-space} assumption 
\begin{equation}
\label{eq:confined-space}
\lim_{M \to \infty} N_M = \infty, ~~ \lim_{M \to \infty} \frac{M}{N^2_M} = \infty, \text{ and } \liminf_{M \to \infty} \frac{y_M}{M} >0.
\end{equation}
Let $\cExp$ denote a unit-mean exponential random variable.
Here is the limit theorem in this case; 
in contrast to
Theorem~\ref{thm:meagre-capacity},
 the initial location~$x_M$ is unimportant for the limit in~\eqref{eq:lifetime-confined-space},
and the initial energy $y_M$ only enters through the lower bound in~\eqref{eq:confined-space}.
 
\begin{theorem}
\label{thm:confined-space}
Consider the $(N_M,M,z_M)$ model with $M \in \N$, $N_M \in \N$ and $z_M \in I_{N_M} \times I_M$
such that~\eqref{eq:confined-space} holds.
Then, 
as $M \to \infty$,
\begin{align}
\label{eq:lifetime-confined-space}
 \frac{4 \lambda}{N_M^2}  \cos^M ( \pi / N_M) & \tod \cExp.
\end{align}
	\end{theorem}
	
	\begin{remarks}
	\phantomsection
\label{rems:confined-space}
\begin{myenumi}[label=(\alph*)]
\item
\label{rems:confined-space-a}
The appearance of the exponential distribution in the limit~\eqref{eq:lifetime-confined-space} is a
consequence of the fact that it is a rare event for the random walk to spend time $\gg N^2$ in the interior of the interval $I_N$,
and can be viewed as a manifestation of the \emph{Poisson clumping heuristic} for Markov chain hitting times~\cite[\S B]{aldous} or \emph{metastability} of the system that arises from the fact
extinction is certain but unlikely on any individual excursion.
\item
\label{rems:confined-space-b}
Since $\log \cos \theta = - (\theta^2/2) + O (\theta^4)$ as $\theta \to 0$, 
we can re-write~\eqref{eq:lifetime-confined-space}, in the case where $M$ does not grow too fast compared to $N_M^2$, as
\[  \frac{4\lambda}{N_M^2} \exp \left\{ - \frac{\pi^2 M}{2N_M^2} \right\}   \tod \cExp, \text{ if } \lim_{M \to \infty} \frac{M}{N_M^4} = 0.\]
\end{myenumi}
\end{remarks}

\subsection{The critical case}
\label{sec:critical}

Finally, we treat the \emph{critical case} in which there exists $\rho \in (0,\infty)$ such that
\begin{equation}
\label{eq:critical}
\lim_{M \to \infty} N_M = \infty, \text{ and } \lim_{M \to \infty} \frac{M}{N^2_M} = \rho.
\end{equation}
Define the decreasing function $H: (0,\infty) \to \RP$ by
\begin{equation}
\label{eq:H-def}
H (y) :=  \sum_{k=1}^\infty h_k(y) , \text{ where } h_k (y) := \exp \left\{ - \frac{\pi^2 (2k-1)^2 y}{2} \right\} .\end{equation}
Since $H(y) \sim 1/\sqrt{8\pi y}$ as $y \downarrow 0$ (see Lemma~\ref{lem:H} below),
 for every $\rho >0$ and $s \in \R$, 
\begin{equation}
\label{eq:G-def}
 G(\rho, s) := \frac{s}{H(\rho)}   \int_0^1 \re^{s v} \bigl( H (v \rho) - H(\rho) \bigr) \ud v , \end{equation}
is finite. For fixed $\rho>0$, $s \mapsto G(\rho,s)$ is strictly increasing for $s \in \R$, and $G (\rho, 0 ) = 0$.
For $\rho >0$, define $s_\rho := \sup \{ s > 0 : G(\rho, s) < 1 \}$, and then set
\begin{equation}
\label{eq:phi-rho-def}
\phi_\rho (s) := \frac{1}{1-G(\rho, s)} , \text{ for } s < s_\rho.\end{equation}
Finally, define $\mu : (0, \infty) \to \RP$ by
\begin{equation}
\label{eq:mu-def}
\mu (\rho) := 
\frac{1}{\rho H(\rho)} \int_0^\rho H(y) \ud y -1 .
\end{equation}
Here is our result in the critical case. For simplicity of presentation,
we restrict to the case where $\zeta_0 = (1,M)$; one could permit initial
conditions similar to those in~\eqref{eq:meagre-capacity},   
but this would complicate the statement and lengthen the proofs (see Remarks~\ref{rems:critical}\ref{rems:critical-d} below).

\begin{theorem}
\label{thm:critical}
Consider the $(N_M,M,z_M)$ model with $M \in \N$, $N_M \in \N$ and $z_M = (1,M)$
such that~\eqref{eq:critical} holds.
Then, 
as $M \to \infty$,
\begin{align}
\label{eq:lifetime-critical}
 \frac{\lambda}{M}  & \tod 1 + \xi_\rho , \text{ and } \frac{\ExpfMo \lambda}{M} \to 1 + \mu (\rho), 
\end{align}
where $\xi_\rho$ has moment generating function $\Exp [ \re^{s \xi_\rho} ] = \phi_\rho (s)$, $s < s_\rho$,
and expectation $\Exp \xi_\rho = \mu (\rho)$. Moreover,
\begin{equation}
\label{eq:mu-asymptotics}
 \lim_{\rho \downarrow 0} \mu (\rho ) = 1, \text{ and } \mu (\rho ) = \frac{\re^{\pi^2 \rho/2}}{4 \rho} (1+o(1)), \text{ as } \rho \to \infty .\end{equation}
	\end{theorem}

\begin{remarks}
	\phantomsection
\label{rems:critical}
\begin{myenumi}[label=(\alph*)]
\item
\label{rems:critical-a}
The function~$H$ defined at~\eqref{eq:H-def} is 
in the family of theta functions, which arise throughout
the distributional theory of one-dimensional Brownian motion on an interval, which is the origin of the appearance here: see e.g.~\cite[pp.~340--343]{feller2} or~Appendix~1 of~\cite{bosa}, and Section~\ref{sec:two-sided} below.
\item
\label{rems:critical-b}
The $\rho \to \infty$ asymptotics in~\eqref{eq:mu-asymptotics}
are consistent with Remark~\ref{rems:confined-space}\ref{rems:confined-space-b}, in the sense that both are consistent with, loosely speaking, the claim that
\[ \ExpfMo \lambda \sim \frac{N_M^2}{4} \exp \left\{ - \frac{\pi^2 M}{2N_M^2} \right\} , \text{ when }   N_M^2 \ll M \ll N_M^4 , \]
although neither of those results formally establishes this.
\item
\label{rems:critical-c}
An integration by parts shows that $G$ defined at~\eqref{eq:G-def} has the representation
\[  G(\rho, s) = \int_0^\infty (  \re^{sx} -1) \frac{m_\rho(x)}{x} \ud x, \text{ where } m_\rho (x) := \frac{\pi^2 \rho x}{2 H(\rho)} \mathbbm{1}_{[0,1]} (x) \sum_{k=1}^\infty (2k-1)^2 h_k (\rho x), \]
which shows that $\phi_\rho$ corresponds to an infinitely divisible distribution (see e.g.~\cite[p.~91]{svh});
since $m_\rho$ is compactly supported, however, the distribution is \emph{not} compound exponential~\cite[p.~100]{svh}.
\item
\label{rems:critical-d}
It is possible to extend
Theorem~\ref{thm:critical} to initial states $z_M = (x_M, y_M)$  
satisfying $y_M/M \to u \in (0,1]$ and $x_M^2/M \to a \in [0,\infty]$,
and  the limit statement~\eqref{eq:lifetime-critical} would need
to be modified to include  additional terms similar to in Theorem~\ref{thm:meagre-capacity}, but with $\tbmo$ replaced by a two-sided exit time. 
We do not pursue this extension here.
\end{myenumi}
\end{remarks}

\subsection{Organization and heuristics behind the proofs}
\label{sec:heuristics}

The rest of the paper presents proofs of Theorems~\ref{thm:meagre-capacity}, \ref{thm:confined-space}, and \ref{thm:critical}.
The basic ingredients are
 provided by some results on excursions of simple symmetric random walk on $\Z$,
given in Section~\ref{sec:random-walk}; some of this material is classical,
but the core estimates that we need are, in parts, quite intricate and we were unable to find them in the
literature.
The main body of the proofs is presented in Section~\ref{sec:proofs}. 
First, we establish a renewal framework that provides the structure for the proofs, which,  
together with a generating-function analysis involving the excursion estimates from Section~\ref{sec:random-walk},
yields the results. In Section~\ref{sec:conclusion} we comment on some possible 
future directions. Appendix~\ref{sec:kummer} briefly presents some necessary facts about the Darling--Mandelbrot distribution appearing in Theorem~\ref{thm:meagre-capacity}. Here
 we outline the main ideas behind the proofs, and make some links to the literature.

It is helpful to first imagine a walker that is ``immortal'', i.e., has an unlimited supply of energy.
The energy-constrained walker is indistinguishable from the immortal walker until the first moment that the time since its most recent visit to the boundary exceeds~$M+1$, at which point the energy-constrained walker
becomes extinct. Let $s_1 < s_2 < \cdots$ denote the successive times of visits to $\bI_N$ by the immortal walker, and let $u_k = s_k - s_{k-1}$
denote the duration of the $k$th excursion ($k \in \N$, with $s_0 :=0$). The energy-constrained walker starts $\kappa \in \N$
excursions before becoming extinct, where $\kappa = \inf \{ k \in \N : u_k > M+1\}$. At every boundary visit at which the energy-constrained walk is still active,
there is a probability $\theta (N, M) = \bP_1 ( \tau_{0,N} > M+1 )$ that the walk will run out of energy on the next excursion,
where $\tau_{0,N}$ is the hitting time of $\bI_N$ by the random walk, and $\bP_1$ indicates we start from site $1$
(equivalently, site $N-1$). Each time the walker visits $\bI_N$ there is a ``renewal'', because energy is topped up to level $M$ and the next excursion begins, started from one step away from the boundary.
If the walk starts at time $0$ in a state other than at full energy, next to the boundary, then the very first excursion has a different distribution from the rest, and this plays a role in some of our results,
but   is a second-order consideration for the present heuristics. 
The key consequence of the renewal structure is that the excursions have a (conditional) independence property, and the number $\kappa$ of excursions has
a geometric distribution with parameter given by the extinction probability $\theta (N,M)$ (see Lemma~\ref{lem:excursion-law} below for a formal statement).

Suppose first that $N = \infty$, the simplest case of the meagre-capacity limit.  
We indicate the relevance of Darling's result on the maxima of stable variables~\cite[Theorem~5.1]{darling} to this model,
which gives some intuition for the appearance of the $\DM$ distribution in Theorem~\ref{thm:meagre-capacity}. First we describe Darling's result. 
Suppose that $Z_1, Z_2, \ldots$ are i.i.d.~$\RP$-valued random variables in the domain of attraction
of a (positive) stable law with index~$\alpha \in (0,1)$, and let $S_n := \sum_{i=1}^n Z_i$ and $T_n := \max_{1 \leq i \leq n} Z_i$. 
Darling's theorem says that 
$S_n / T_n$ converges in distribution to $1 +\xi_\alpha$, as $n \to \infty$, where $\xi_\alpha \sim \mathrm{DM} (\alpha)$.
A generalization to other order statistics is~\cite[Corollary~4]{ab}.

In the case where $N = \infty$, the durations $u_1, u_2, \ldots$ of excursions of simple symmetric random walk on $\ZP$ away from~$0$
satisfy the $\alpha =1/2$ case of Darling's result, so that $T_n:= \sum_{i=1}^n u_i$ and $M_n:= \max_{1 \leq i \leq n} u_i$ satisfy
$T_n / M_n \tod 1 + \xi$ where $\xi \sim \DM$. Replacing~$n$ by $\kappa$, the number of excursions up to extinction,
for which $\kappa \to \infty$ in probability as $M, N \to \infty$, it is plausible that $T_\kappa / M_\kappa \tod 1 + \xi$ also.
But $T_\kappa$ is essentially $\lambda$, while $M_\kappa$ will be very close to~$M$, the upper bound on $u_i$, $i < \kappa$.
This heuristic argument is not far from a proof of the case $N = \infty$ of Theorem~\ref{thm:meagre-capacity}. 
More precisely, one can express $\lambda$ in terms of the \emph{threshold sum process}~\cite{bs}, and then the $N=\infty$, $a=0$~case of Theorem~\ref{thm:meagre-capacity} is a consequence of Theorem~1 of~\cite{bs}.
Another way to access intuition behind the $M$-scale result for $\lambda$ is that in this case $\theta (N, M)$ is of order $M^{-1/2}$ (see Proposition~\ref{prop:excursion-characteristics} below),
so there are of order $M^{1/2}$ completed excursions before extinction,
while the expected duration of an excursion of length less than $M$ is of order $M^{1/2}$ (due to the $1/2$-stable tail).  
Our proof below (Section~\ref{sec:proofs:small-energy}) covers the full regime $M \ll N^2$. That this is the relevant scale is due to the fact that simple random walk travels distance about $\sqrt{M}$ in time $M$,
so if $N^2 \gg M$ it is rare for the random walk to encounter the opposite end of the boundary from which it started.

Consider now the confined-space regime, where $M \gg N^2$. 
Now it is very likely that the random walk will traverse the whole of $I_N$ many times before it runs out of energy, and so there will be 
 many excursions before extinction.
Indeed, the key quantitative result in this case, Proposition~\ref{prop:confined-space-excursions} below,
shows that $\theta (N, M) \sim (4/M) \cos^M (\pi/N)$, which is small. 
Each excursion has mean duration about $N$ ($\bE_1 \tau_{0,N} = N-1$; see Lemma~\ref{lem:mean-tau}).
Roughly speaking,  the law of large numbers ensures that $\lambda \approx \kappa N$, and then
\[ \Pr ( \lambda \geq n ) \approx \Pr ( \kappa \geq n/N ) \approx (1 -\theta (N,M) )^{n/N} \approx \exp \left\{ - \frac{4n}{N^2} \cos^M \left( \frac{\pi}{N} \right) \right\} ,\]
which is essentially the exponential convergence result in Theorem~\ref{thm:confined-space}.

The case that is most delicate is the critical case where $M \sim \rho N^2$. The extinction probability estimate, given in Proposition~\ref{prop:critical-excursions} below,
is now $\theta (N, M) \sim (4/N) H(\rho)$, where $H$ is given by~\eqref{eq:H-def};
the delicate nature is because, on the critical scale, the two-boundary nature of the problem has an impact (unlike the meagre-capacity regime),
while extinction is sufficiently likely that the largest individual excursion fluctuations are on the same scale as the total lifetime
(unlike the confined-space regime). 
Since $\theta (N, M)$ is of order $1/N$,
 both the number of excursions before extinction, and the duration of a typical excursion, are of order $N$ (i.e., $M^{1/2}$),
similarly to the meagre-capacity case, and so again there is an $M$-scale limit for $\lambda$, but the (universal) Darling--Mandelbrot
distribution is replaced by the curious distribution exhibited in Theorem~\ref{thm:critical}, which we have not seen elsewhere.

\section{Excursions of simple random walks}
\label{sec:random-walk}

\subsection{Notation and preliminaries}
\label{sec:rw-notation}

For $x \in \Z$, let $\bP_x$ denote the law of a simple symmetric random walk on $\Z$ with initial state~$x$.
We denote by $S_0, S_1, S_2, \ldots$ the trajectory of the random walk with law $\bP_x$, realised on a suitable probability space, so that $\bP_x (S_0 = x) =1$
and $\bP_x ( S_{n+1} - S_n = + 1 \mid S_0, \ldots, S_n ) = \bP_x ( S_{n+1} - S_n = - 1 \mid S_0, \ldots, S_n) = 1/2$ for all $n \in \ZP$.
Let $\bE_x$ denote the expectation corresponding to $\bP_x$.
We sometimes write simply $\bP$ and $\bE$ in the case where the initial state plays no role.

For $y \in \Z$, let $\tau_y := \inf \{ n \in \ZP : S_n = y \}$, the hitting time of~$y$. As usual, we set $\inf \emptyset := +\infty$;
 the recurrence of the random walk says that $\bP_x ( \tau_y < \infty ) = 1$ for all $x, y \in \Z$.
Also define 
$\tau_{0,\infty} := \tau_0$ and, for $N \in \N$,
$\tau_{0,N} := \tau_0 \wedge \tau_N = \inf \{ n \in \ZP : S_n \in \{0,N\}\}$.

The number
of $(2n+1)$-step simple random walk paths that start at $1$ and visit $0$ for the first time at time $2n+1$ is
the same as the number of $(2n+1)$-step paths that start at $0$, finish at $1$, and never return to $0$,
which, by the classical ballot theorem~\cite[p.~73]{feller1}, is $\frac{1}{2n+1} \binom{2n+1}{n+1} = \frac{1}{n+1} \binom{2n}{n}$.
Hence, by Stirling's formula (cf.~\cite[p.~90]{feller1}),
\begin{equation}
\label{eq:tau_0-mass}
 \bPo ( \tau_0 = 2n + 1) = \frac{2^{-1-2n}}{n+1} \binom{2n}{n} \sim \frac{1}{2 \sqrt{ \pi}} n^{-3/2} , \text{ as } n \to \infty.\end{equation}
Similarly, since $\bPo ( \tau_0 \geq 2n+1 ) = \bPo ( S_1 >0, \ldots, S_{2n-1} > 0 )
= 2\bP_0 ( S_1 > 0, \ldots, S_{2n} > 0)$ (cf.~\cite[pp.~75--77]{feller1}) we have that
\begin{equation}
\label{eq:tau_0-tail} \bPo ( \tau_0 \geq 2n + 1) 
= 2^{-2n} \binom{2n}{n} 
\sim \frac{1}{\sqrt{\pi n}}, \text{ as } n \to \infty  .\end{equation}
The distribution of $\tau_{0,N}$ is more complicated; there's an exact formula (see e.g.~\cite[p.~369]{feller1})
that will be needed when we look at larger time-scales (see Theorem~\ref{thm:feller-tail-general} below),
but
for shorter time-scales, it will suffice to approximate the two-sided exit time $\tau_{0,N}$
in terms of the (simpler) one-sided exit time $\tau_0$. This is the subject of the next subsection.

\subsection{Short-time approximation by one-sided exit times}
\label{sec:one-sided}

The next lemma studies the duration of excursions that are constrained to be short.

\begin{lemma}
\phantomsection
\label{lem:one-sided-approx}
\begin{enumerate}[label=(\roman*)]
\item\label{lem:one-sided-approx-i}
For all $N \in \barN$,
all $x \in I_N$, and all $n \in \ZP$, 
\begin{equation}
\label{eq:one-sided-approx-1}
  \left|  \bP_x ( \tau_{0,N} > n ) - \bP_x ( \tau_{0} > n ) \right| \leq \frac{x}{N}  .\end{equation}
\item\label{lem:one-sided-approx-ii}
	Let $\eps >0$. Then there exists $n_\eps \in \N$, depending only on $\eps$, such that, for all $N \in \barN$,
\begin{equation}
\label{eq:one-sided-approx-2}
  \left|  \bPo ( \tau_{0,N} > n ) - n^{-1/2} \sqrt{ {2} / {\pi}}  \right| \leq \eps n^{-1/2} + \frac{1}{N} , \text{ for all }  n \geq n_\eps .\end{equation}
\item\label{lem:one-sided-approx-iii}
Fix $\eps >0$. Then there exist $n_\eps \in \N$ and $\delta_\eps \in (0,\infty)$ such that
\begin{equation}
\label{eq:fast-return-tail}
\sup_{n_\eps \leq n < \delta_\eps N^2 } \left| n^{1/2} \bP_1 ( \tau_{0,N} > n ) - \sqrt{ {2} / {\pi}}   \right| \leq \eps, \text{ for all } N \in \barN .
\end{equation}
\end{enumerate}
\end{lemma}
\begin{proof}
Suppose that $N \in \N$ and $x \in I_N$.
Since $\tau_0 \neq \tau_{0,N}$ if and only if $\tau_N < \tau_0$,
 \[ \sup_{n \in \ZP} \left| \bP_x ( \tau_{0,N} > n ) -  \bP_x ( \tau_0 > n ) \right| \leq \bP_x ( \tau_N < \tau_0 ) = \frac{x}{N}, \text{ for all } x \in I_N,\]
by the classical gambler's ruin result for symmetric random walk;
if $N = \infty$ then $\bP_x ( \tau_{0,N} > n) = \bP_x (\tau_0 > n)$ by definition. This verifies~\eqref{eq:one-sided-approx-1}.
Then~\eqref{eq:one-sided-approx-2} follows from
the $x=1$ case of~\eqref{eq:one-sided-approx-1} together with~\eqref{eq:tau_0-tail}. Finally, for a given $\eps>0$, the bound~\eqref{eq:one-sided-approx-2} implies that there exists $n_\eps \in \N$
such that, for all $n_\eps \leq n \leq \delta N^2$,
\[  \left|   \bP_1 ( \tau_{0,N} > n ) -  n^{-1/2} \sqrt{{2}/{\pi}}  \right|  \leq  ( \eps + \delta^{1/2} ) n^{-1/2} \leq 2 \eps n^{1/2} , \]
if $\delta = \eps^2$, say. This yields~\eqref{eq:fast-return-tail} (suitably adjusting~$\eps$).
\end{proof}

Lemma~\ref{lem:first-excursion} gives a  limit result for the duration of a simple random walk excursion started with initial condition $x_M$ of order $\sqrt{M}$;
we will apply this to study the initial excursion of the energy-constrained walker. 
Recall that $b = (b_t, t\in \RP)$ denotes standard Brownian motion started at $b_0 =0$, and $\tbmo$~its first time of hitting level~$1$.

\begin{lemma}
\label{lem:first-excursion}
Let $a \in [0,\infty]$. Suppose that $M, N_M \in \N$ and $x_M \in I_{N_M}$ are such that $x^2_M / M \to a$ and $x_M / N_M \to 0$ as $M \to \infty$. 
Then for every $y \in \RP$,
\begin{equation}
\label{eq:limit-distribution-truncated}
\frac{\tau_{0,N_M}}{M} \1 { \tau_{0,N_M} \leq y M } \tod a \tbmo \1 { a \tbmo \leq y } ,\end{equation}
where the right-hand side of~\eqref{eq:limit-distribution-truncated} is to be interpreted as $0$ whenever $a \in \{0,\infty\}$.
In particular,  
for every $\beta \in \RP$ and all $y \in \RP$,
\begin{align} 
\label{eq:limit-moments-truncated}
 \lim_{M \to \infty} M^{-\beta} \bE_{x_M} [ \tau^\beta_{0,N_M} \1 { \tau_{0,N_M} \leq y M } ]  = a^\beta \Exp [  (\tbmo)^\beta \1 { a \tbmo \leq y } ].
\end{align}
\end{lemma}
\begin{proof}
Suppose that $0 \leq a < \infty$.
Let $\cC$ denote the space of continuous functions $f : \RP \to \R$, endowed with the uniform metric
$\| f - g \|_\infty := \sup_{t \in \RP} | f(t) - g(t) |$.
Define the re-scaled and interpolated random walk trajectory $z_M \in \cC$ by 
\[ z_M (t) =    M^{-1/2} \left( S_{\lfloor M t \rfloor} - S_0 + (Mt - \lfloor M t \rfloor) ( S_{\lfloor M t \rfloor +1} - S_{\lfloor M t \rfloor} ) \right)  , \text{ for } t \in \RP. \]
Then $S_0 / M^{1/2} \to \sqrt{a}$, and, 
by Donsker's theorem (see e.g.~Theorem 8.1.4 of~\cite{durrett}), $z_M$
converges weakly in $\cC$ to $(b_t)_{t \in \RP}$, where $b$ is standard Brownian motion started from~$0$. 
For $f \in \cC$, let $T_a (f) := \inf \{ t \in \RP : f(t) \leq -\sqrt{a} \}$.
By Brownian scaling, $T_a (b)$ has the same distribution as $a \tbmo$.
With probability~1, $T_a(b)< \infty$ and $b$ has no intervals of constancy. 
Hence (cf.~\cite[p.~395]{durrett}) the set $\cC' := \{ f \in \cC : T_a \text{ is finite and continuous at } f \}$  has $\Pr ( b \in \cC' ) = 1$
and hence, 
 by the continuous mapping theorem, $T_a ( z_M) \to T_a (b)$ in distribution. If $a = \infty$ this says $T_\infty (z_M) \to \infty$ in probability,
and if $a=0$ it says  $T_0 (z_M) \to 0$  in probability;
otherwise, $\Pr ( a \tbmo > y )$ is continuous for all $y \in \RP$ and $\tau_0 / M = T_a ( z_M + M^{-1/2} S_0 - \sqrt{a} )$ has the same limit
as $T_a (z_M)$. Hence we conclude that 
\begin{equation}
\label{eq:weak-limit}
 \lim_{M \to \infty} \bP_{x_M} ( \tau_0 > y M) = \Pr ( a \tbmo > y ) , \text{ for all } y \in \RP,\end{equation}
where the right-hand side is equal to~$0$ if $a=0$ and $1$ if $a=\infty$.
It follows from~\eqref{eq:one-sided-approx-1} that~\eqref{eq:weak-limit} also holds for $\tau_{0,N_M}$ in place of $\tau_0$,
provided that $x_M / N_M \to 0$. Hence, under the conditions of the lemma, we have
\begin{equation}
\label{eq:limit-distribution}
  \frac{\tau_{0,N_M}}{M} \tod a \tbmo , \text{ as } M \to \infty.\end{equation}
If random variables $X, X_n$ satisfy $X_n \tod X$, then, for every $y \in \RP$,
 $X_n \1 { X_n \leq y } \tod X \1{ X \leq y }$;
this follows from the fact that
\[ \Pr ( X_n \1 { X_n \leq y } \leq x ) = \begin{cases} 1 & \text{if } x \geq y ,\\
\Pr (X_n \leq x)  & \text{if } x < y , \end{cases} \]
and $x <y$ is a continuity point of $\Pr ( X \1{ X \leq y } \leq x)$
if and only if it is a continuity point of $\Pr (X \leq x)$. Hence~\eqref{eq:limit-distribution}
implies~\eqref{eq:limit-distribution-truncated}. The bounded convergence theorem then yields~\eqref{eq:limit-moments-truncated}.
\end{proof}

\subsection{Long time-scale asymptotics for two-sided exit times}
\label{sec:two-sided}

The following result gives the 
expectation and variance of the duration of the classical gambler's ruin game;
the expectation can be found, for example, in~\cite[pp.~348--349]{feller1},
while the variance is computed in~\cite{bach,ah}.

\begin{lemma}
\label{lem:mean-tau}
For every $N \in \N$ and every $x \in I_N$, we have
\[ \bE_x \tau_{0,N} = x (N -x), \text{ and } \bVar_x \tau_{0,N} = \frac{x (N-x)}{3}  \left[ x^2 + (N-x)^2 -2 \right] .\]
\end{lemma}

Note that although $\bEo \tau_{0,N} = N-1$,  $\bVar_1 \tau_{0,N} \sim N^3/3$ is much greater than the square of the mean,
which reflects the fact that while (under $\bPo$) $\tau_{0,N}$ is frequently very small,
with probability about $1/N$ it exceeds $N^2$ (consider reaching around $N/2$ before $0$).
Theorem~\ref{thm:feller-tail-general} below gives asymptotic estimates for the tails $\bPo ( \tau_{0,N} > n )$.
These estimates are informative when $n$ is at least of order $N^2$, 
i.e.,
at least the scale for $\tau_{0,N}$ that contributes to most of the variance. 

Define the trigonometric sum
\begin{equation}
\label{eq:S-0-def} 
\cS_0 (N , m , n)   := \sum_{k=1}^{m}  \cos^{n} \left( \frac{\pi (2k-1)}{N} \right).
\end{equation}
 
\begin{theorem}
\label{thm:feller-tail-general}
Suppose that $k_0 \in \N$. Then, there exists $N_0 \in \N$ (depending only on $k_0$) such that, 
for all $N \geq N_0$ and all $n \in \ZP$,
\begin{equation}
\label{eq:feller-tail-general} \bPo ( \tau_{0,N} > n ) = \frac{4}{N}  \bigl[ 1 + \Delta  (N, k_0,  n)  \bigr] \cS_0 ( N , k_0, n) , \end{equation}
where the function  $\Delta$ satisfies
\begin{align}
\label{eq:Delta-bound} 
| \Delta (N, k_0, n ) | & \leq \frac{4 \pi^2 k_0^2}{N^2} + 2 \left( 1 +  \frac{N^2}{4 \pi^2 n k_0 } \right)  \exp \left\{ - \frac{2 \pi^2 n k_0^2}{N^2} \right\} .\end{align}
\end{theorem}

Before giving the proof of Theorem~\ref{thm:feller-tail-general}, we state  
two important consequences. Part~\ref{cor:feller-tail-i} of Corollary~\ref{cor:feller-tail}
will be our key estimate in studying the confined space regime ($M \gg N^2$), while part~\ref{cor:feller-tail-ii} 
is required for the critical regime ($M \sim \rho N^2$). 

\begin{corollary}
\phantomsection
\label{cor:feller-tail}
\begin{enumerate}[label=(\roman*)]
\item\label{cor:feller-tail-i}
If $n / N^2 \to \infty$ as $n \to \infty$, then
\begin{equation}
\label{eq:feller-tail-supercritical}  
\bPo ( \tau_{0,N} > n ) = \frac{4}{N} (1 + o(1) ) \cos^n ( \pi / N) .
\end{equation}
Moreover, for every $\beta \in \RP$,
\begin{equation}
\label{eq:large-n-moments}
 \lim_{n \to \infty} \frac{\bEo [ \tau^\beta_{0,N} \mid \tau_{0,N} \leq n ]}{\bEo [ \tau^\beta_{0,N} ]} = 1 .\end{equation}
\item\label{cor:feller-tail-ii} 
Let $y_0 \in (0,\infty)$. Then, with $H$ defined in~\eqref{eq:H-def},
\[ \lim_{N \to \infty} \sup_{y \geq y_0} \left| \frac{N}{4} \bPo ( \tau_{0,N} > y N^2 ) - H(y) \right| = 0 .\]
\end{enumerate}
\end{corollary}

\begin{remark}
\label{rem:cor-feller-tail}
We are not able to obtain the  conclusions of Corollary~\ref{cor:feller-tail}
 directly from Brownian scaling arguments. Indeed, 
an appealing approximation is to say $\bPo ( \tau_{0,N} > n ) \approx \Pr_{1/N} ( \tbmoo > n/N^2)$, where $\tbmoo$ is the first hitting time of $\{0,1\}$
for Brownian motion started at~$1/N$. This approximation does not, however, achieve the correct asymptotics
for the full range of $n , N$, even if the error is quantified. In particular
(see e.g.~\cite[p.~342]{feller2} or~\cite[p.~126]{bosa}) we have
\[ \Pr_{1/N} ( \tbmoo > n/N^2) = \frac{4}{\pi} \sum_{m=0}^\infty \frac{\sin \bigl( \frac{(2m+1) \pi}{N}  \bigr)}{2m+1} \exp \left \{ - \frac{ (2m+1)^2 \pi^2 n }{2 N^2} \right \} ,\]
which for $n \gg N^2$ leads to the conclusion that
\[ \Pr_{1/N} ( \tbmoo > n/N^2) = \frac{4}{N} (1 + o(1)) \exp \left\{ -\frac{ \pi^2  n}{2 N^2 } \right\} .\]
 This agrees with asymptotically with~\eqref{eq:feller-tail-supercritical} only when $N^2 \ll n \ll N^4$; cf.~Remark~\ref{rems:confined-space}\ref{rems:confined-space-b}.
Hence we develop the quantitative estimates in Theorem~\ref{thm:feller-tail-general}.
\end{remark}

To end this section, we complete the proofs of Theorem~\ref{thm:feller-tail-general} and Corollary~\ref{cor:feller-tail}.

\begin{proof}[Proof of Theorem~\ref{thm:feller-tail-general}.]
A classical result, whose origins Feller traces to Laplace~\cite[p.~353]{feller1}, yields
\begin{align}
\label{eq:feller-laplace}
 \bPo ( \tau_{0,N} = n ) & = \frac{1}{N} \sum_{k=1}^{N-1} \cos^{n-1} \left( \frac{\pi k}{N} \right)  \sin \left( \frac{\pi k}{N} \right) \left[ \sin \left( \frac{\pi k}{N} \right)
+ \sin \left( \frac{\pi (N-1) k}{N} \right) \right] \nonumber\\
& = \frac{2}{N}  \sum_{k=1}^{N-1} \sin^2 \left( \frac{\pi k}{2} \right) \cos^{n-1} \left( \frac{\pi k}{N} \right)  \sin^2 \left( \frac{\pi k}{N} \right)  \nonumber\\
& = \frac{2}{N}  \sum_{k=1}^{\left\lceil \frac{N-1}{2} \right\rceil} \cos^{n-1} \left( \frac{\pi (2k-1)}{N} \right)  \sin^2 \left( \frac{\pi (2k-1)}{N} \right) .
\end{align}
(The expression in~\eqref{eq:feller-laplace} is $0$ if $N$ and $n$ are both even.)
Define $m_N := \left\lceil \frac{N-1}{2} \right\rceil$ and note that $N - 2m_N = \delta_N \in \{0,1\}$, where
\begin{equation}
\label{eq:delta-def}
\delta_N := \begin{cases} 0 &\text{if $N$ is even}, \\
1 & \text{if $N$ is odd.} \end{cases} \end{equation}
Also define 
\begin{equation}
\label{eq:S-1-def} 
\cS_1 (N , m , n)   := \sum_{k=1}^{m}  \cos^{n} \left( \frac{\pi (\delta_N+2k-1)}{N} \right) .
\end{equation}
Then, from~\eqref{eq:feller-laplace}, and the notation introduced in~\eqref{eq:S-0-def},
we have
\begin{align}
\label{eq:feller-tail}
 \bPo ( \tau_{0,N} > n ) 
& = \frac{2}{N}  \sum_{k=1}^{m_N} \frac{\cos^{n} \left( \frac{\pi (2k-1)}{N} \right)  \sin^2 \left( \frac{\pi (2k-1)}{N} \right) }{1 -\cos \left( \frac{\pi (2k-1)}{N} \right)}  \nonumber\\
& = \frac{2}{N}  \sum_{k=1}^{m_N} \left\{ 1 + \cos \left( \frac{\pi (2k-1)}{N} \right) \right\} \cos^{n} \left( \frac{\pi (2k-1)}{N} \right)  \nonumber\\
& = \frac{2}{N} \cS_0 ( N, m_N,  n) + \frac{2}{N} \cS_0 (N, m_N , n+1).
\end{align}
The proof of the theorem requires several key estimates. The first claim is that
\begin{align}
\label{eq:S-0-bound-in-k}
\cS_0 ( N , m_N , n ) & = \cS_0 ( N , k_0, n)  +   (-1)^n \cS_1 ( N, k_0 , n ) + \Delta_1 ( N, k_0, n ),\end{align}
for all $m_N > k_0 \in \N$,
where $\cS_1$ is defined at~\eqref{eq:S-1-def}, and
\[ | \Delta_1 ( N, k_0, n  ) | \leq 2 \left( 1 +  \frac{N^2}{4 \pi^2 n k_0 } \right)  \exp \left\{ - \frac{\pi^2 n (2k_0+1)^2}{2N^2} \right\}  .\]
The second claim is that, for every $k_0 \in \N$,
\begin{equation}
\label{eq:bound-n-to-n+1}
\begin{split}
\left| \cS_0 ( N , k_0, n+1)- \cS_0 ( N , k_0, n) \right| & \leq \frac{4 \pi^2 k_0^2}{N^2}  \cS_0 ( N , k_0, n) ;\\
\left| \cS_1 ( N , k_0, n+1)- \cS_1 ( N , k_0, n) \right| & \leq \frac{4 \pi^2 k_0^2}{N^2}  \cS_0 ( N , k_0, n). 
\end{split}
\end{equation}
Take the bounds in~\eqref{eq:S-0-bound-in-k} and~\eqref{eq:bound-n-to-n+1} as given, for now;
then from~\eqref{eq:feller-tail} and~\eqref{eq:S-0-bound-in-k}, 
\begin{align*}
 \bPo ( \tau_{0,N} > n )  & = \frac{2}{N} \bigl[ \cS_0 ( N , k_0, n) +  \cS_0 ( N , k_0, n+1) \bigr] \\
& {} \qquad {} + 
\frac{2}{N} (-1)^n \bigl[ \cS_1 ( N, k_0 , n ) - \cS_1 ( N, k_0 , n+1 ) \bigr]  \\ 
& {} \qquad {} + \frac{2}{N} \bigl[ \Delta_1 ( N,   k_0, n )+  \Delta_1 ( N, k_0, n+1 ) \bigr],
\end{align*}
and then applying~\eqref{eq:bound-n-to-n+1} we obtain
\begin{equation}
\label{eq:feller-tail-general-2} 
\bPo ( \tau_{0,N} > n ) = \frac{4}{N}  \bigl[  ( 1 + \Delta_2 (N, k_0, n) )  \cS_0 ( N , k_0, n) + \Delta_3 ( N,  k_0, n)   \bigr], \end{equation}
where the terms $\Delta_2, \Delta_3$ satisfy
\begin{align}
\label{eq:Delta-2-bound} 
| \Delta_2 ( N, k_0, n ) | & \leq \frac{4 \pi^2 k_0^2}{N^2} ; \\
\label{eq:Delta-3-bound} 
| \Delta_3 (N, k_0, n  ) | & \leq  2 \left( 1 +  \frac{N^2}{4 \pi^2 n k_0 } \right)  \exp \left\{ - \frac{ \pi^2 n (2k_0+1)^2}{2N^2} \right\} .\end{align}
Now, there exists $\theta_0 > 0$ ($\theta_0=1$ will do) for which $\log \cos \theta \geq - \theta^2$ for all $| \theta | \leq \theta_0$.
Hence
\[ \cS_0 (N, k_0, n) \geq \cos^n ( \pi/N) \geq \exp \left\{ -\frac{\pi^2 n}{N^2} \right\}, \]
for all $N$ sufficiently large (as a function of $k_0 \in \N$), and all $n \in \N$. Hence, by~\eqref{eq:Delta-3-bound},
\[ \frac{| \Delta_3 ( N, k_0, n)|}{\cS_0 (N, k_0, n) } \leq  2 \left( 1 +  \frac{N^2}{4 \pi^2 n k_0 } \right)  \exp \left\{ - \frac{\pi^2 n ( ( 2 k_0+1)^2 -2 )}{2N^2} \right\},\]
and, together with~\eqref{eq:feller-tail-general-2}  and~\eqref{eq:Delta-2-bound}, this implies~\eqref{eq:feller-tail-general}.
 
It remains to verify~\eqref{eq:S-0-bound-in-k} and~\eqref{eq:bound-n-to-n+1}. For the former, write
\begin{align*}
\cS_0 (N , m , n) & = \sum_{k=1}^{\lfloor m/2 \rfloor }  \cos^{n} \left( \frac{\pi (2k-1)}{N} \right) + \sum_{k=\lfloor m/2 \rfloor+1}^{m}  \cos^{n} \left( \frac{\pi (2k-1)}{N} \right) \\
& = \sum_{k=1}^{\lfloor m/2 \rfloor }  \cos^{n} \left( \frac{\pi (2k-1)}{N} \right) + (-1)^n \sum_{\ell=1}^{\lceil m/2 \rceil}   \cos^{n} \left( \frac{\pi (N - 2m + 2\ell-1)}{N} \right) ,\end{align*}
using the change of variable $\ell = m-k+1$ and the fact that $\cos ( \pi - \theta ) = -\cos \theta$.
Since $\log \cos \theta \leq - \theta^2/2$ for $| \theta | < \pi/2$, we have
\begin{equation}
\label{eq:cos-bound-1} 0 \leq \cos^{n} \left( \frac{\pi (2k-1)}{N} \right)  
\leq \exp \left\{ - \frac{\pi^2 n (2k-1)^2}{2N^2} \right\}, \text{ for } 1 \leq k \leq N/4 ,\end{equation}
and, similarly, since $N - 2m_N = \delta_N \in \{0,1\}$, 
\begin{equation}
\label{eq:cos-bound-2} 0 \leq  \cos^{n} \left( \frac{\pi (\delta_N + 2\ell-1)}{N} \right) \leq \exp \left\{ - \frac{\pi^2 n (2\ell-1)^2}{2N^2} \right\}, \text{ for } 1 \leq \ell \leq N/4 .\end{equation}
From~\eqref{eq:cos-bound-1}, we obtain that, for any $k_0 \in \N$, since $\lfloor m_N /2 \rfloor \leq N/4$,
\begin{align*} \sum_{k=k_0+1}^{\lfloor m_N/2 \rfloor }  \cos^{n} \left( \frac{\pi (2k-1)}{N} \right) 
& \leq \sum_{k=k_0+1}^{\lfloor m_N/2 \rfloor } \exp \left\{ - \frac{\pi^2 n (2k-1)^2}{2N^2} \right\} \\
& \leq \sum_{\ell = 0}^{\infty} \exp \left\{ - \frac{\pi^2 n (2\ell + 2k_0 +1)^2}{2N^2} \right\} \\
& \leq \exp \left\{ - \frac{\pi^2 n (2k_0+1)^2}{2N^2} \right\} \sum_{\ell = 0}^{\infty} \exp \left\{ - \frac{4 (1+k_0) \pi^2 n \ell }{N^2} \right\} ,
\end{align*}
using the inequality $(2\ell + 2k_0 + 1)^2  \geq 8 (1+k_0) \ell + (2k_0 +1)^2$.
Since
\[ \sum_{\ell =0}^\infty \re^{-\alpha \ell} \leq 1 + \int_0^\infty \re^{-\alpha x} \ud x = 1 + \frac{1}{\alpha}, \text{ for } \alpha > 0 ,\]
we get
\[ 0 \leq \sum_{k=k_0+1}^{\lfloor m_N/2 \rfloor }  \cos^{n} \left( \frac{\pi (2k-1)}{N} \right) \leq
 \left( 1 +  \frac{N^2}{4 \pi^2 n k_0 } \right) \exp \left\{ - \frac{\pi^2 n (2k_0+1)^2}{2N^2} \right\}  .\]
Similarly, from~\eqref{eq:cos-bound-2}, we get
\[ 0 \leq  \sum_{\ell=k_0+1}^{\lceil m_N/2 \rceil}   \cos^{n} \left( \frac{\pi (N - 2m_N + 2\ell-1)}{N} \right) \leq
 \left( 1 +  \frac{N^2}{4 \pi^2 n k_0 } \right)  \exp \left\{ - \frac{\pi^2 n (2k_0+1)^2}{2N^2} \right\} .\]
This verifies~\eqref{eq:S-0-bound-in-k}.
Finally, we have from~\eqref{eq:S-0-def} that
\begin{align*}
\left| \cS_0 ( N , k_0, n+1)- \cS_0 ( N , k_0, n) \right| & \leq \cS_0 ( N , k_0, n) \sup_{1 \leq k \leq k_0} \left| 1 - \cos \left( \frac{\pi (2k -1)}{N} \right) \right|  \\
& \leq \frac{4 \pi^2 k_0^2}{N^2} \cS_0 ( N , k_0, n) , \end{align*}
for all $N$ sufficiently large (depending only on $k_0$), 
since $| 1 - \cos \theta | \leq \theta^2$ for all $\theta \in \R$.
Similarly,
 \begin{align*}
\left| \cS_1 ( N , k_0, n+1)- \cS_1 ( N , k_0, n) \right| & \leq \cS_1 ( N , k_0, n) \sup_{1 \leq k \leq k_0} \left| 1 - \cos \left( \frac{\pi (\delta_N + 2k -1)}{N} \right) \right|  \\
& \leq \frac{4 \pi^2 k_0^2}{N^2} \cS_1 ( N , k_0, n) \leq \frac{4 \pi^2 k_0^2}{N^2} \cS_0 ( N , k_0, n) ,
\end{align*}
where, again, $N$ must be large enough. This verifies~\eqref{eq:bound-n-to-n+1}.
\end{proof}

\begin{proof}[Proof of Corollary~\ref{cor:feller-tail}.]
First, suppose that $n / N^2 \to \infty$. Then we can take $k_0 =1$ in Theorem~\ref{thm:feller-tail-general} to see that
$\bPo ( \tau_{0,N} > n ) = (4/N) (1+ \Delta  (N, 1,  n) ) \cos^n ( \pi/N)$,
where, by~\eqref{eq:Delta-bound},
$| \Delta (N,1,n) | = o(1)$. This proves~\eqref{eq:feller-tail-supercritical}.
Since $\log \cos \theta \leq - \theta^2/2$ for $|\theta | < \pi/2$,
it follows from~\eqref{eq:feller-tail-supercritical} that
\begin{equation}
\label{eq:feller-tail-supercritical-bound}
 \bPo ( \tau_{0,N} > n)  \leq \frac{5}{N} \exp \left\{ - \frac{\pi^2 n}{2N^2} \right\}, \text{ for all $n$ large enough.} \end{equation}
For any random variable $X \geq 0$, and $\beta \in \RP$, and any $r \in \RP$, one has
\begin{align*}
 \Exp [ X^\beta \1 { X \geq r } ] 
& = 
\int_0^\infty \Pr ( X^\beta \1 { X \geq r } > y ) \ud y \\
& \leq r^\beta \Pr ( X \geq r) + \int_{r^\beta}^\infty \Pr ( X^\beta > y ) \ud y\\
& =  r^\beta \Pr ( X \geq r) + \beta \int_{r}^\infty s^{\beta -1} \Pr ( X > s ) \ud s .
\end{align*}
Hence~\eqref{eq:feller-tail-supercritical-bound} shows that there is a constant $C_\beta \in \RP$ such that, for all $n$ sufficiently large,
\begin{align*}
 \bEo [ \tau^\beta_{0,N} \1 { \tau_{0,N} \geq n+1 } ] 
& \leq \frac{C_\beta n^\beta}{N} \exp \left\{ - \frac{\pi^2 n}{2 N^2} \right\} + \frac{C_\beta}{N} \int_{n}^\infty  s^{\beta-1} \exp \left\{ - \frac{\pi^2 s}{4 N^2} \right\} \ud s \\
& =  C_\beta N^{2\beta -1} \left[ \left( \frac{n}{N^2} \right)^\beta \exp \left\{ - \frac{\pi^2 n}{2 N^2} \right\} +   \int_{n/N^2}^\infty  t^{\beta-1} \exp \left\{ - \frac{\pi^2 t}{4} \right\} \ud t \right],
\end{align*}
using the change of variable $t = s/N^2$. Since $n /N^2 \to \infty$,
 this means that
\begin{equation}
\label{eq:big-tau-bound}
 \bEo [ \tau^\beta_{0,N} \1 { \tau_{0,N} \geq n+1 } ]  = o ( N^{2\beta -1}).
\end{equation}
In particular, the $\beta =0$ case of~\eqref{eq:big-tau-bound} shows that $\lim_{n \to \infty} \bPo ( \tau_{0,N} \leq n ) =1$.
A consequence of~\eqref{eq:fast-return-tail} is that
$\bPo ( \tau_{0,N} > \delta N^2 ) \geq \delta/N$ for some $\delta >0$, and hence
$\bEo [ \tau^\beta_{0,N} ] \geq c_\beta N^{2\beta -1}$
for some $c_\beta >0$; the conclusion in~\eqref{eq:large-n-moments} then follows from~\eqref{eq:big-tau-bound}. The proves part~\ref{cor:feller-tail-i}.

Finally, fix $y \geq y_0 > 0$ and suppose that $n / N^2 \to y$. Then, uniformly in $k \leq N^{1/2}$, 
\begin{equation}
\label{eq:critical-asymptotic}
 \cos^n \left( \frac{\pi (2k-1)}{N} \right) = \exp \left\{ - \frac{\pi^2 n (2k-1)^2}{2N^2} (1+O(k^2/N^2)) \right\} = (1+o(1)) h_k (y) ,\end{equation}
 as $n \to \infty$, where $h_k$ is defined at~\eqref{eq:H-def}
and $\sum_{k=1}^\infty h_k (y) = H(y)$ satisfies $\sup_{y \geq y_0} H(y) < \infty$. 
In particular, for any $\eps >0$ we can choose $k_0 \in \N$ large enough so that $\sup_{y \geq y_0} |\sum_{k=1}^{k_0} h_k (y) - H(y) | < \eps$.
From~\eqref{eq:critical-asymptotic} it follows that, for any $\eps>0$, 
\[   \left|  {\sum_{k= 1}^{k_0}  \cos^n \left( \frac{\pi (2k-1)}{N} \right) } - {\sum_{k=1}^{k_0}  h_k (y) }   \right| \leq \eps, \]
for  all $n$ sufficiently large. Hence,
for every $\eps>0$,
\begin{equation}
\label{eq:critical-asymptotic-1}
\sup_{y \geq y_0} \left|  \cS_0 ( N, k_0, n )   - H(y) \right| \leq \eps ,\end{equation}
for all $n$ sufficiently large. 
By~\eqref{eq:Delta-bound},
there exist $k_0, n_0 \in \N$ such that, for every $y \geq y_0$,
\begin{equation}
\label{eq:critical-asymptotic-2}
 | \Delta (N, k_0, n ) |   \leq \frac{4 \pi^2 k_0^2}{N^2} + 2 \left( 1 +  \frac{1}{3 \pi^2 y k_0 } \right)  \exp \left\{ - \pi^2 k_0^2 y \right\} \leq \eps,  \end{equation}
for all $n \geq n_0$ (given $\eps$ and $y_0$, first take $k_0$ large, and then $N$ large).
From~\eqref{eq:feller-tail-general} with~\eqref{eq:critical-asymptotic-1}, \eqref{eq:critical-asymptotic-2}, and the fact that 
$\sup_{y \geq y_0} H(y) < \infty$, we verify part~\ref{cor:feller-tail-ii}.
\end{proof}

\section{Proofs of main results}
\label{sec:proofs}

\subsection{Excursions, renewals, and extinction}
\label{sec:proofs:general}

We start by giving the proof of the irreducibility result, Lemma~\ref{lem:irreducibility}, stated in Section~\ref{sec:finite-capacity}.

\begin{proof}[Proof of Lemma~\ref{lem:irreducibility}.]
Let $\cF_n := \sigma ( \zeta_0, \ldots, \zeta_n)$, the $\sigma$-algebra generated by the first~$n \in \ZP$ steps of the Markov chain~$\zeta$.
Then, given $\cF_n$, at least one of the two neighbouring sites of~$\Z$ to $X_n$ is in $\iI_N$; take $y \in \iI_N$ to be any site such that $| y - X_n | = 1$.
If $X_n \in \bI_N$, then the walker will move to $y$ with probability 1, and the energy level will be refreshed to $M$:
\begin{equation}
\label{eq:exponential-bound-1}
 \Pr ( X_{n+1}  \in \iI_N, \, \eta_{n+1} = M  \mid \cF_n ) = 1, \text{ on } \{ X_n \in \bI_N \}. \end{equation}
Otherwise, $X_n \in \iI_N$. If $\eta_n \geq 1$, then
the walker can, with probability $1/2$, take a step to $y$ on the next move, which uses 1 unit of energy. On the other hand, if $\eta_n = 0$, then $X_{n+1} = X_n$ must remain in $\iI_N$.
Thus, writing $x^+ := x \1 { x > 0}$,
\begin{equation}
\label{eq:exponential-bound-2}
\Pr ( X_{n+1} \in \iI_N, \, \eta_{n+1} = (\eta_n -1)^+  \mid \cF_n ) \geq \frac{1}{2}, \text{ on } \{ X_n \in \iI_N \}. \end{equation}
In particular, since $\eta_n \leq M$, we can combine~\eqref{eq:exponential-bound-1}
and~\eqref{eq:exponential-bound-2} (applied $M$ times) to get
\[ \Pr ( \eta_{(k+1)(M+1)} = 0, \, X_{(k+1)(M+1) } \in \iI_N \mid \cF_{k(M+1)} ) \geq 2^{-M}, \as , \text{ for all } k \in \ZP .\]
Thus 
\[ \Pr ( \lambda > (k+1) (M+1) \mid \cF_{k(M+1)} ) \leq (1 - 2^{-M}) \1 { \lambda > k (M+1) } , \as , \]
repeated application of which implies that
\[ \sup_{z \in I_N \times I_M} \Prf ( \lambda > k (M+1) ) \leq (1-2^{-M} )^k \leq \exp \left\{ - k 2^{-M} \right\} , \]
uniformly over $N \in \barN$. Every $n \in \ZP$ has $k (M+1) \leq n < (k+1)(M+1)$ for some $k = k(n) \in \ZP$, and so
\[ \Prf (\lambda > n ) \leq \Prf (\lambda > k (M+1))  \leq \exp \left\{ - \left(\frac{2^{-M}}{M+1}\right) n \right\} , \text{ for all } n \in \ZP.\]
The verifies that $\sup_{z \in I_N \times I_M} \Expf  [ \re^{\delta \lambda} ] < \infty$ for any $\delta \in (0,  \frac{2^{-M}}{M+1} )$.
\end{proof}

We denote by $\sigma_1 < \sigma_2 < \cdots < \sigma_\kappa$ the successive times of visiting the boundary before time~$\lambda$;
formally, set $\sigma_0 := 0$ and
\[ \sigma_k := \inf \{ n > \sigma_{k-1} : X_n \in \bI_N \}, \text{ for } k \in \N , \]
with the usual convention that $\inf \emptyset := \infty$. 
Then $\sigma_k < \infty$ if and only if $\sigma_k < \lambda$.
The \emph{number of complete excursions} is then 
\begin{equation}
\label{eq:kappa-def}
\kappa := \max \{ k \in \ZP : \sigma_k < \lambda \},
\end{equation} 
i.e., the number of visits to $\bI_N$ before extinction. Since $\lambda < \infty$, a.s.~(by Lemma~\ref{lem:irreducibility}),
$\kappa \in \ZP$ is a.s.~finite.

For $k \in \N$, define the \emph{excursion duration} 
$\nu_k := \sigma_k - \sigma_{k-1}$ as long as $\sigma_k <\infty$; otherwise, set $\nu_k = \infty$.
We claim that
\begin{equation}
\label{eq:nu-range}
 \text{for every } k \in \N, ~ \nu_k \in I_{M+1} \cup \{ \infty \}, \text{ and } \nu_k < \infty \text{ if and only if } \sigma_k < \lambda .
\end{equation}
To see~\eqref{eq:nu-range}, observe that,
 if $\sigma_{k-1} < \lambda$ then $\eta_{\sigma_{k-1}+1} = M$
and $X_n \in \iI_N$ for all $\sigma_{k-1} < n < \sigma_k$. Hence $\eta_{\sigma_{k-1} + i} = M+1 -i$
for all $1 \leq i \leq \nu_k$ if $\nu_k < \infty$. In particular, 
if $\nu_k < \infty$, then 
$0 \leq \eta_{\sigma_k} = M + 1 - \nu_k$, which implies that $\nu_k \leq M+1$. This verifies~\eqref{eq:nu-range}.

For simplicity, we write $\nu := \nu_1$.
By the strong Markov property, for all $k , n \in \ZP$,
\begin{equation}
\label{eq:conditional-iid}
\Prf ( \nu_{k+1} = n \mid \cF_{\sigma_k} ) = \Prfo ( \nu = n), \text{ on } \{ \sigma_k < \lambda \} ,
\end{equation}
meaning that excursions subsequent to the first are identically distributed. On the other hand,
$\Prf ( \nu_1 = n ) = \Prf ( \nu = n)$ for $n \in \ZP$, so that if $z \neq (1,M)$
the first excursion may have a different distribution.

After the final visit to the boundary at time $\sigma_\kappa$, the energy~$\eta_{\sigma_\kappa+1} =M$ decreases one unit at a time until $\eta_{\sigma_\kappa + M + 1} = 0$
achieves extinction. Thus, by~\eqref{eq:nu-range}, 
\begin{equation}
\label{eq:lambda-sigma}
\lambda = M+1 + \sigma_\kappa = M + 1 + \sum_{k=1}^{\kappa} \nu_k.
 \end{equation}
In the terminology of renewal theory,  $\nu_1, \nu_2, \ldots$ is a renewal sequence that is
\emph{delayed} 
(since $\nu_1$ may have a different distribution from the subsequent terms)
and \emph{terminating}, since $\Pr (\nu_k = \infty) >0$. 
A key quantity is the per-excursion \emph{extinction probability}
\begin{equation}
\label{eq:theta-def}
 \theta_z ( N, M ) :=  \Prf ( \lambda < \sigma_1 ) = \Prf ( \nu_1 = \infty ) ,\end{equation}
the probability that the process started in state~$z$ terminates before reaching the boundary. Set $\theta (N,M) := \theta_{(1,M)} (N,M)$
for the case where $z=(1,M)$, which plays a special role, due to~\eqref{eq:conditional-iid}. The following basic result
exhibits the probabilistic structure associated with the renewal set-up. Recall the definitions of the number of completed excursions~$\kappa$ and extinction probability $\theta_z(N,M)$ from~\eqref{eq:kappa-def}
and~\eqref{eq:theta-def}, respectively. 

\begin{lemma}
\label{lem:excursion-law}
Let $N \in \barN$ and $M \in \N$. 
Then, for all $k \in \N$, 
\begin{equation}
\label{eq:kappa-geometric}
 \Prf ( \kappa \geq k ) =  (1- \theta_z (N,M) ) (1 - \theta (N, M) )^{k-1} , \end{equation}
and
\[ \Expf [ \kappa ] = \frac{1- \theta_z (N,M)}{\theta (N,M)} .\]
Moreover, given that $\kappa  =k \in \N$, the random variables $\nu_1, \ldots \nu_k$
are (conditionally) independent, and satisfy, for every $n_1, \ldots, n_k \in I_{M+1}$,
\begin{align}
\label{eq:conditional-joint-law}
& {} \Prf ( \nu_1 = n_1, \ldots, \nu_k = n_k \mid \kappa = k ) \nonumber\\
& {} \quad {} = \Prf ( \nu = n_1 \mid \nu < \infty ) \prod_{i=2}^k \Prfo ( \nu = n_i \mid  \nu < \infty ) .
\end{align}
\end{lemma}
\begin{proof}
By definition, $\sigma_0 = 0$.
Let $k \in \N$. From~\eqref{eq:kappa-def}, we have that $\kappa \geq k$ if and only if $\sigma_k < \lambda$.
If $\sigma_k < \lambda$, then $|X_{\sigma_k+1} - y | = 1$ for some $y \in \bI_N$,
and $\eta_{\sigma_k+1} = M$.
Hence, by~\eqref{eq:theta-def} and the strong Markov property applied at the stopping time $\sigma_k+1$,
for $k \in \N$,  
\[ \Prf ( \kappa \geq k+1 \mid \cF_{\sigma_k} ) = \Prf ( \sigma_{k+1} < \lambda \mid \cF_{\sigma_k} ) = 1- \theta (N, M), \text{ on } \{ \lambda > \sigma_k \} .\]
Hence
$\Prf ( \kappa \geq k+1 \mid \kappa \geq k ) = 1 - \theta (N,M)$ for $k \in \N$, and, since $\Prf ( \kappa \geq 1) = \Prf ( \sigma_1 < \lambda  ) = 1 - \theta_z (N,M)$, we obtain~\eqref{eq:kappa-geometric}. 
Moreover, for $n_1, \ldots, n_k \leq M+1$,
\begin{align}
\label{eq:conditional-joint-law1}
& {} \Prf \left( \{ \kappa = k \} \cap ( \cap_{i=1}^k \{ \nu_i = n_i \} ) \right) \nonumber\\
& {} \quad {} = \Prf \left( \{ \nu_{k+1} = \infty \} \cap ( \cap_{i=1}^k \{ \nu_i = n_i \} ) \right) \nonumber\\
& {} \quad {} = \Prf (\nu = n_1) \Prfo ( \nu = \infty ) \prod_{i=2}^k \Prfo ( \nu = n_i ) ,\end{align}
by repeated application of~\eqref{eq:conditional-iid}. Similarly,
\begin{align}
\label{eq:conditional-joint-law2}
\Prf ( \kappa = k ) & = \Prf  \left( \{ \nu_{k+1} = \infty \} \cap ( \cap_{i=1}^k \{ \nu_i < \infty \} ) \right) \nonumber\\
& = \Prf (\nu < \infty ) \Prfo ( \nu = \infty ) \prod_{i=2}^k \Prfo ( \nu < \infty ) . \end{align}
Dividing the expression in~\eqref{eq:conditional-joint-law1} by that in~\eqref{eq:conditional-joint-law2}
gives~\eqref{eq:conditional-joint-law}, since, by~\eqref{eq:theta-def},
$ \Prf ( \nu < \infty ) = 1 - \theta_z (N,M)$.
\end{proof}

\subsection{The meagre-capacity limit}
\label{sec:proofs:small-energy}

In this section we present the proof of Theorem~\ref{thm:meagre-capacity}; at several
points we appeal to the results from Section~\ref{sec:random-walk}
on simple random walk. Consider the generating function
\begin{equation}
\label{eq:phi-def}
\Mgf (s) := \sum_{n=1}^{M+1} \re^{sn} \PrfMo ( \nu = n \mid \nu < \infty ) , \text{ for } s \in \R.\end{equation}
Define for $t \in \R$,
\begin{equation}
\label{eq:kummer}
 K(t):= 1 - \sum_{\ell=1}^\infty \frac{t^\ell}{(2\ell-1) \cdot \ell!}, \text{ and }  \cT := \{ t \in \R : K(t) > 0 \}.
\end{equation}
For each $t \in \R$, the series in~\eqref{eq:kummer} converges absolutely, and, indeed, 
$| 1 - K(t) | \leq \re^{|t|}$ for all $t \in \R$.
The series for $K$ defined by~\eqref{eq:kummer} 
compared to equation~(13.1.2) in~\cite{as} identifies $K$ as the Kummer (confluent hypergeometric) function $K(t) = M(-\frac{1}{2}, \frac{1}{2} , t )$;
see Appendix~\ref{sec:kummer} for some of its properties.

The following result gives asymptotics
for $\theta (N_M, M) = \theta_{(1,M)} (N_M, M)$ as defined at~\eqref{eq:theta-def},
and for the generating function~$\Mgf$ as defined at~\eqref{eq:phi-def}.

\begin{proposition}
\label{prop:excursion-characteristics}
Consider the $(N_M,M,z_M)$ model with $M \in \N$ and $N_M \in \barN$ such that
$\lim_{M \to \infty} (M / N_M^2 ) = 0$.
Then 
\begin{equation}
\label{eq:theta-asymptotics}
 \theta (N_M, M ) = \sqrt{\frac{2}{\pi}} (1+o(1)) M^{-1/2}  , \text{ as } M \to \infty . \end{equation}
Moreover,
\begin{equation}
\label{eq:fast-return-mgf}
 \lim_{M \to \infty} \left[ \left( \Mgf (t / M) -1 \right) {M^{1/2}} \right] = \sqrt{ \frac{2}{\pi} } (1 - K(t))  ,\end{equation}
uniformly over $t \in R$, for any compact $R \subset \R$.
\end{proposition}

In the proof of this result, and later, we will use the following integration by parts formula for restricted expectations,
which is a slight generalization of Theorem~2.12.3 of~\cite[p.~76]{gut}.

\begin{lemma}
For any real-valued random variable $X$ and every $a, b \in \R$ with $a \leq b$, and any monotone and differentiable $g: \R \to \R$,
\begin{align}
\label{eq:mgf-split-by-parts}
 \Exp [ g (X) \1 { a < X \leq b} ] & =  g(a) \Pr (X > a ) - g(b) \Pr (X > b) + \int_a^b g' (y) \Pr (X > y) \ud y . \end{align}
\end{lemma}
\begin{proof}
Suppose that $g: \R \to \R$ is differentiable.
Write $F_X (y):=\Pr ( X \leq y)$ for the distribution function of $X$. If $g$ is monotone non-decreasing, then
\begin{align*}
 \Exp [ g (X) \1 { a < X \leq b} ] = \int_a^b g (y) \ud F_X (y)  = g(b) F_X(b) - g(a) F_X (a) - \int_a^b F_X (y ) g' (y) \ud y, \end{align*}
where the first equality is e.g.~Theorem~2.7.1 of~\cite[p.~60]{gut} and the second follows
from Theorem~2.9.3 of~\cite[p.~66]{gut}; this yields~\eqref{eq:mgf-split-by-parts}.
If $g$ is monotone non-increasing, then the same result follows by considering $\tilde g (y) := g(b) - g(y)$.
\end{proof}

\begin{proof}[Proof of Proposition~\ref{prop:excursion-characteristics}.]
For ease of notation, write $\tau := \tau_{0,N_M}$ throughout this proof. 
Apply Lemma~\ref{lem:one-sided-approx} with $n = M+1$ to get
\[ \theta (N_M, M ) = \bPo ( \tau > M+1 ) = \sqrt{\frac{2}{\pi}}  M^{-1/2} + O (1/N_M) , \text{ as } M \to \infty . \]
Since $M = o( N_M^2 )$, the $O(1/N_M)$ term here can be neglected
asymptotically; this verifies~\eqref{eq:theta-asymptotics}. 

We apply~\eqref{eq:mgf-split-by-parts} with $X = \tau/M$, $a=0$, $b=1$, and $g (y) = \re^{ty} -1$, ($t \neq 0$) to get
\begin{align}
\label{eq:mgf-by-parts-meagre}
 & \bEo [ (\re^{t \tau / M} - 1) \1 { \tau \leq M } ] \nonumber\\
& {} \qquad {} = (1 - \re^{t}) \bPo ( \tau  > M ) + t \int_0^1 \re^{tu} \bPo ( \tau  > u M ) \ud u . \end{align}
Fix $\eps >0$, and let $n_\eps \in \N$ and $\delta_\eps > 0$ be as in Lemma~\ref{lem:one-sided-approx}\ref{lem:one-sided-approx-iii}.
Since $M / N_M^2 \to 0$, we have $M < \delta_\eps N_M^2$ for all $M$ sufficiently large. Hence from~\eqref{eq:fast-return-tail} we have that
\begin{equation}
\label{eq:uniform-bound-meagre}
 \sup_{n_\eps/M \leq u \leq 1} \left| M^{1/2} u^{1/2} \bPo ( \tau  > u M ) - \sqrt{2/\pi} \right| \leq \eps ,\end{equation}
for all $M$ sufficiently large. It follows from~\eqref{eq:uniform-bound-meagre} that, as $M \to \infty$, 
\begin{align}
\label{eq:mgf-kummer-lower}
 t \int_0^1 \re^{tu} \bPo ( \tau  > u M ) \ud u
& \geq \bigl( \sqrt{2/\pi}  - \eps \bigr) M^{-1/2} t \int_{n_\eps/M}^1 \re^{tu} u^{-1/2} \ud u \nonumber\\
&  =  \bigl( \sqrt{2/\pi}  - \eps \bigr) M^{-1/2} ( \cI( t) + o(1) ) ,\end{align}
using the definition of $\cI$ from~\eqref{eq:E-def};
here the $o(1)$ is uniform for $t \in R$ for a given compact $R \subset \R$.
 Similarly, from~\eqref{eq:uniform-bound-meagre} again, uniformly for $t \in R$ (compact),
\begin{align}
\label{eq:mgf-kummer-upper}
 t \int_0^1 \re^{tu} \bPo ( \tau  > u M ) \ud u
& \leq \bigl( \sqrt{2/\pi}  + \eps \bigr) M^{-1/2} ( \cI( t) + o(1) ). \end{align}
Combining~\eqref{eq:mgf-kummer-lower} and~\eqref{eq:mgf-kummer-upper}, since $\eps>0$ was arbitrary, we obtain
\[ t \int_0^1 \re^{tu} \bPo ( \tau  > u M ) \ud u = \bigl( \sqrt{2/\pi} +o(1) \bigr) M^{-1/2} \cI( t) ,\]
as $M \to \infty$,
uniformly for $t \in R$, $R$ compact. 
 Together with the asymptotics for $\bPo ( \tau  > M+1 ) = \theta (N_M, M )$ from~\eqref{eq:theta-asymptotics},
we thus conclude from~\eqref{eq:mgf-by-parts-meagre} that
\[ \bEo [ (\re^{t \tau  / M} - 1) \1 { \tau  \leq M+1 } ] = \bigl( \sqrt{2/\pi} +o(1) \bigr) M^{-1/2} \left( 1 - \re^t +  \cI( t) \right) , \]
uniformly for $t \in R$, $R$ compact, from which~\eqref{eq:fast-return-mgf} follows, since
$1 - \re^t +  \cI( t) = 1-K(t)$ by~\eqref{eq:K-E}, and, from~\eqref{eq:phi-def},
\[ \Mgf (t / M)   = 1 + \ExpfMo [ \re^{t \nu  / M} - 1 \mid \nu < \infty ] = \frac{\bEo [ (\re^{t \tau  / M} - 1) \1 { \tau  \leq M+1 } ]}{\bPo ( \tau \leq M+1) } ,\]
where $\bPo ( \tau \leq M+1) = 1 - \theta (N_M, M ) \to 1$, by~\eqref{eq:theta-asymptotics}.
\end{proof}

To avoid burdening notation with conditioning, 
let $Y_M$ denote a random variable taking values in $I_{M+1}$ (enriching the underlying probability space if necessary)
such that
\begin{equation}
\label{eq:Y-def}
 \PrfMo ( Y_M = n )  = \PrfMo ( \nu = n \mid \nu < \infty ), \text{ for } n \in I_{M+1} .\end{equation}
Given $Y_n$ as constructed through~\eqref{eq:Y-def}, it is the case that $\Mgf$ as defined at~\eqref{eq:phi-def} 
can be represented via 
$\Mgf (s) = \ExpfMo [ \re^{s Y_M} ]$, $s \in \R$.

The next lemma proves that the limit distribution in Theorem~\ref{thm:meagre-capacity}
has expectation given by the function~$g$ defined at~\eqref{eq:g-def}, and
establishes some of the key properties of~$g$.

\begin{lemma}
\label{lem:g-calculus}
Suppose that $\xi \sim \DM$ and $\tbmo \sim \Stable$ are independent with distributions given by~\eqref{eq:xi-transform} and~\eqref{eq:tbm} respectively. 
Let $a \in \RP$ and $u \in (0,\infty)$. The random variable $\zeta_{a,u} := \min ( u, a \tbmo ) + (1+ \xi) \1 { a \tbmo < u}$ has mean given by
$\Exp \zeta_{a,u} = g(a,u)$, where $g$ is defined at~\eqref{eq:g-def}. Moreover, for every $u\in (0,1]$,
the function~$a \mapsto g(a,u)$
 is strictly decreasing and infinitely differentiable on $(0,\infty)$, with~$g(0,u)=2$ and $\lim_{a \to \infty} g(a,u)=u$.
\end{lemma}
\begin{proof}
If $a=0 <u$, then $\zeta_0 = 1 +\xi$, a.s., and the result is true because $g(0,u) = 2 = 1 +\Exp \xi$.
Suppose that $a, u \in (0,\infty)$.
From~\eqref{eq:tbm}, we have that
\begin{align*}
\Exp [ u \wedge a \tbmo ]
& = \int_0^{u/a} \frac{a}{\sqrt{2\pi t}} \re^{-1/(2t)} \ud t + u \int_{u/a}^\infty \frac{1}{\sqrt{2\pi t^3}} \re^{-1/(2t)} \ud t \\
& = \sqrt{\frac{au}{2\pi}} \int_{1}^\infty  y^{-3/2}  \re^{-ay/(2u)} \ud y + \frac{2u}{\sqrt{2\pi}} \int_0^{\sqrt{a/u}} \re^{-s^2 /2} \ud s , \end{align*}
using the substitutions $y = u/(at)$ in the first integral and $s^2 = 1/t$ in the second.
Here
\[  \frac{2}{\sqrt{2\pi}} \int_0^{\sqrt{a/u}} \re^{-s^2 /2} \ud s =  2 \bigl[ \Phi ( \sqrt{a/u} ) - \Phi (0) \bigr] = 2 \Phi ( \sqrt{a/u} )  - 1 .\]
Moreover,  an integration by parts followed by the substitution $s^2 = ry$ gives
\begin{align*}
 \int_{1}^\infty  y^{-3/2}  \re^{-r y/2} \ud y & = 2 \re^{-r/2} - r \int_1^\infty y^{-1/2} \re^{-ry/2} \ud y \\
& = 2 \re^{-r/2} - 2 \sqrt{r} \int_{\sqrt{r}}^\infty \re^{-s^2/2} \ud s , \text{ for } r > 0 . \end{align*}
Hence
\[ \Exp [ u \wedge a \tbmo ] = u + \sqrt{\frac{2au}{\pi}} \re^{-a/(2u)} - 2 (a+u) \bPhi ( \sqrt{a/u} ) .\]
Finally, by~\eqref{eq:reflection}, $\Pr (  a \tbmo < u ) = \Pr ( \tbmo < u/a) = 2 \bPhi (\sqrt{a/u} )$,
and then using the independence of $\xi$ and $\tbmo$, and the fact that $\Exp \xi = 1$, gives
\[ \Exp \zeta_{a,u} =u + \sqrt{\frac{2au}{\pi}} \re^{-a/(2u)} - 2 (a+u) \bPhi ( \sqrt{a/u} ) + 4  \bPhi ( \sqrt{a/u} ), \]
which is equal to~$g(a,u)$ as defined at~\eqref{eq:g-def}.

Write $\phi (z) := (2\pi)^{-1/2} \re^{-z^2/2}$ for the standard normal density, so that $\bPhi' ( z) = - \phi(z)$
and $\phi'(z) = - z \phi(z)$. The formula~\eqref{eq:g-def}
can then be expressed as 
\[ g(a) = u + (4-2u-2a) \bPhi (\sqrt{a/u} ) + 2 \sqrt{au} \phi (\sqrt{a/u}), \]
which,
on differentiation, yields, for $u >0$,
\begin{align} 
\label{eq:g-derivative}
g' (a,u) & := \frac{\partial}{\partial a} g (a,u) 
= -\frac{2(1-u)}{\sqrt{au}} \phi (\sqrt{au} ) - 2 \bPhi (\sqrt{a/u} ) .\end{align} 
Thus 
$g'(a,u) < 0$ for all $a \in (0,\infty)$ and all $u \in (0,1]$,
with $\lim_{a \to 0} g'(a,u) = -\infty$ for $u \in (0,1)$ and
$\lim_{a \to 0} g'(a,1) = -1$. In particular, $a \mapsto g(a,u)$ is strictly decreasing.
\end{proof}

\begin{proof}[Proof of Theorem~\ref{thm:meagre-capacity}.]
To simplify notation, set 
$\theta_M := \theta (N_M, M)$.
First suppose that $X_0 = 1$ and $\eta_0 =M$.
Then in the representation $\sigma_\kappa = \sum_{i=1}^\kappa \nu_i$, 
Lemma~\ref{lem:excursion-law} shows that, given $\kappa = k \in \N$, $\nu_1, \ldots, \nu_k$ are i.i.d.~with the law of $Y_M$ as given by~\eqref{eq:Y-def},
and the law of $\kappa$ is $\PrfMo ( \kappa = k )  = (1-\theta_M)^k \theta_M$, for $k \in \ZP$.
In particular, for $| r (1-\theta_M ) | < 1$,
\[ \ExpfMo [ r^{\kappa} ] =\sum_{k=0}^\infty r^k (1-\theta_M)^k \theta_M = \frac{\theta_M}{1-(1-\theta_M)r} .\]
 Hence, by (conditional) independence,
\begin{align}
\label{eq:Xi-gf-s}
 \ExpfMo \left[ {\re^{s \sigma_\kappa}} \right] 
& = \ExpfMo \left[ \Exp \biggl[ \prod_{i=1}^{\kappa} \re^{s \nu_i} \biggmid \kappa \biggr] \right] \nonumber\\
& = \ExpfMo \left[ \left( \Mgf (s) \right)^{\kappa} \right] \nonumber\\
& = \frac{\theta_M}{1-(1-\theta_M) \Mgf (s)} , \text{ if } (1 -\theta_M) \Mgf (s) < 1. 
\end{align}
For $s=t/M$, we have from Proposition~\ref{prop:excursion-characteristics} that, for $t \in \R$,
\begin{equation}
\label{eq:rate-bound} 
( 1 - \theta_M ) \Mgf ( t / M ) = 1 - M^{-1/2} \sqrt{\frac{2}{\pi}}  K(t) + o (M^{-1/2} ) ,\end{equation}
as $M \to \infty$. In particular, for $t \in \cT$, with $\cT$ defined at~\eqref{eq:kummer}, the asymptotics in~\eqref{eq:rate-bound} show that $( 1 - \theta_M ) \Mgf ( t / M ) < 1$
for all $M$ sufficiently large. Hence from~\eqref{eq:Xi-gf-s} applied at  $s = t/M$, $t \in \cT$, 
with~\eqref{eq:rate-bound}  and another application of~\eqref{eq:theta-asymptotics}, we get
\begin{align}
\label{eq:Xi-gf-t}
\lim_{M \to \infty} \ExpfMo \left[ {\re^{t \sigma_\kappa /M}} \right] 
& = \lim_{M \to \infty} \left[ M^{1/2} \sqrt{\frac{\pi}{2}}  \frac{\theta_M}{K(t) +o(1)} \right]
= \frac{1}{K(t)}, \text{ for } t \in \cT .  
\end{align}
Since $\lambda = M +1 + \sigma_\kappa$, by~\eqref{eq:lambda-sigma}, it follows from~\eqref{eq:Xi-gf-t}, 
and the fact that convergence of moment generating functions in a neighbourhood of $0$ implies convergence in distribution, and convergence of all moments (see e.g.~\cite[p.~242]{gut}), 
that
$\lambda /M \tod 1+\xi$ where $\Exp [ \re^{t \xi} ] = 1 / K(t)$ for $t \in \cT$, and $M^{-1} \ExpfMo [ \sigma_\kappa ] \to \Exp \xi = 1$.
The form for $\Exp [ \re^{t \xi}]$ given in~\eqref{eq:xi-transform} is obtained using the relation~\eqref{eq:K-E}. This establishes both~\eqref{eq:lifetime-meagre-limit}
and~\eqref{eq:expectation-limit} 
in the case $X_0 = x_M \equiv 1$ and $\eta_0 = y_M \equiv M$, where~\eqref{eq:meagre-capacity} is satisfied for $a=0$ and $u=1$.

More generally, suppose that $(X_0,\eta_0) = ( x_M,y_M)$ satisfying~\eqref{eq:meagre-capacity}. Then 
\begin{equation}
\label{eq:lambda-general-start}
 \lambda \eqd M + 1 + (\nu_1 + Z ) \1{ \nu_1 < \infty } ,\end{equation}
by the strong Markov property,
where the $Z , \nu_1$ on the right are independent, and $Z$ has the distribution of $\sigma_\kappa$ under $\PrfMo$.
Here $\sigma_\kappa / M \tod \xi$ and, by Lemma~\ref{lem:first-excursion},
\[ \lim_{M \to \infty} \PrfM ( \nu_1 < \infty ) = \lim_{M \to \infty} M^{-1} \bP_{x_M} ( \tau_{0,N_M} \leq y_M + 1 ) = \Pr ( a \tbmo \leq u ) , \]
and $M^{-1} \nu_1 \1{ \nu_1 < \infty }$ converges in distribution to $a \tbmo \1 { a \tbmo \leq u }$. This completes the proof of~\eqref{eq:lifetime-meagre-limit}.
Taking expectations in~\eqref{eq:lambda-general-start}, and using the stated independence,
we get
\begin{align*}
\ExpfM \lambda & = M + 1 + \ExpfM [ \nu  \1{ \nu < \infty } ] +\PrfM ( \nu < \infty)  \ExpfMo [ \sigma_\kappa ]  , \end{align*}
where $\PrfM ( \nu < \infty) = \theta_{z_M} (N_M, M) = \Pr ( a \tbmo \leq u) + o(1)$, by the $\beta =0$ case of~\eqref{eq:limit-moments-truncated}
and hypothesis~\eqref{eq:meagre-capacity}. It follows that 
\begin{align*}
M^{-1} \ExpfM \lambda & = 1 +  M^{-1} \ExpfM [ \nu  \1{ \nu < \infty } ] +  \Pr ( a \tbmo \leq u ) + o(1),\end{align*}
using the fact that $\lim_{M \to \infty} M^{-1} \ExpfMo [ \sigma_\kappa ] = 1$, as established via~\eqref{eq:Xi-gf-t} above. 
Moreover, using~\eqref{eq:meagre-capacity}, we have 
from Lemma~\ref{lem:first-excursion} that
\begin{align*} \lim_{M \to \infty} M^{-1} \ExpfM [ \nu  \1{ \nu  < \infty} ] & = \lim_{M \to \infty} M^{-1} \bE_{x_M} [ \tau_{0,N_M} \1 { \tau_{0,N_M} \leq y_M+1 } ] \\
& = \Exp [ a \tbmo \1 { a \tbmo \leq u} ].\end{align*}
Thus we conclude 
that, as $M \to \infty$, 
\[ \lim_{M \to \infty} M^{-1} \ExpfM \lambda = 1  + a \Exp [ \tbmo \1 { a \tbmo \leq u} ] + \Pr ( a \tbmo \leq u )  = \Exp \zeta_{a,u}, \]
where $\zeta_{a,u}$ is as defined in Lemma~\ref{lem:g-calculus}. This establishes~\eqref{eq:expectation-limit}, and completes the proof of the theorem.
\end{proof}

\subsection{The confined-space limit}
\label{sec:proofs:large-energy}

In this section we present the proof of Theorem~\ref{thm:confined-space}. 
As in the previous section, we start with an asymptotic
estimate on $\theta (N_M, M) := \theta_{(1,M)} (N_M, M)$ as defined at~\eqref{eq:theta-def}.

\begin{proposition}
\label{prop:confined-space-excursions}
Suppose that~\eqref{eq:confined-space} holds.
Then, as $M \to \infty$, 
\begin{equation}
\label{eq:feller-tail-supercritical-theta} \theta ( N_M, M) = \frac{4}{N_M} (1 + o(1)) \cos^M \left( \frac{\pi}{N_M} \right)  . \end{equation}
Moreover, as $M \to \infty$,
\begin{equation}
\label{eq:weak-conditioning-moments}
\ExpfMo [ \nu \mid \nu  < \infty ] \sim N_M, \text{ and }
\VarfMo [ \nu \mid \nu  < \infty ] \sim \frac{N_M^3}{3} .
\end{equation}
\end{proposition}
\begin{proof}
This follows from Corollary~\ref{cor:feller-tail}\ref{cor:feller-tail-i}. Indeed~\eqref{eq:feller-tail-supercritical-theta} is~\eqref{eq:feller-tail-supercritical}  
applied to $\theta (N_M, M ) = \bPo ( \tau_{0,N_M} > M+1 )$. Similarly~\eqref{eq:weak-conditioning-moments}
is a consequence of~\eqref{eq:large-n-moments}
together with the facts that $\bEo \tau_{0,N} = N-1$
and $\bVar \tau_{0,N} = N (N-1)(N-2)/3$, given in Lemma~\ref{lem:mean-tau}.
\end{proof}

We will use the following exponential convergence result for triangular arrays.

\begin{lemma}
\label{lem:exponential-limit}
Let $K_M \in \ZP$ satisfy $\Pr ( K_M = k) = (1-p_M)^{k} p_M$ for $k \in \ZP$, where $p_M \in (0,1)$, and $\lim_{M \to \infty} p_M = 0$.
Suppose also that $Y_M,Y_{M,1}, Y_{M,2}, \ldots$ are i.i.d., $\RP$-valued, and independent of~$K_M$,
with $\Exp [ Y_M^2 ] = \sigma_M^2 < \infty$ and $\Exp Y_M = \mu_M >0$.
Let $Z_M := \sum_{i=1}^{K_M} Y_{M,i}$. Assuming that
\begin{equation}
\label{eq:exponential-parameter-condition}
 \lim_{M \to \infty} \frac{\sigma_M^2 p_M}{\mu_M^2} = 0,\end{equation}
it is the case that, as $M \to \infty$,
\[ \frac{p_M Z_M}{\mu_M} \tod \cExp .\]
\end{lemma}

Lemma~\ref{lem:exponential-limit} can be deduced from e.g. Theorem~3.2.4 of~\cite[p.~85]{kal},
by verifying that the condition~\eqref{eq:exponential-parameter-condition} implies
the `uniformly weighted family' condition from~\cite[p.~44]{kal}. For convenience, we include a direct proof here.

\begin{proof}[Proof of Lemma~\ref{lem:exponential-limit}.]
Let $r_M := 1/p_M \in (0,\infty)$. 
For $s \in (0,\infty)$, we have 
\[ \Exp [ s^{K_M} ] = \frac{p_M}{1-(1-p_M) s} .\]
Write $\psi_M (t) := \Exp [ \re^{ i t Y_M} ]$, the characteristic function of $Y_M$.  
Then, for $t \in \R$,
\begin{align*}
\Exp [ \re^{ i tZ_M} ]
& = \Exp \left[ \Exp \biggl[ \exp \biggl\{ \sum_{j=1}^{K_M} i t Y_{M,j} \biggr\} \biggmid K_M \biggr] \right]  
 = \Exp \left[ (\psi_M (t) )^{K_M} \right] 
 = \frac{p_M}{1-(1-p_M) \psi_M (t) } .\end{align*}
Set $a_M := r_M \mu_M$. 
Thus, for $t \in \R$,
\begin{equation}
\label{eq:char-fn}
 \Exp [ \re^{i t Z_M/a_M} ]  
= \frac{1}{1 - (r_M-1) (\psi_M (t/a_M) -1) } .\end{equation}
Fix $t_0 \in (0,\infty)$.
The $n=1$ case of (3.3.3) in~\cite[p.~135]{durrett},
together with the facts that
$\Exp [ Y_M^2 ] = \sigma_M^2 < \infty$ and $\Exp Y_M = \mu_M$,
yields
\[ r_M \sup_{t \in [-t_0, t_0]} \left|   \psi_M ( t / a_M ) - 1 -  \frac{ i t}{r_M} \right| \leq | t_0 | \frac{\sigma_M^2}{\mu_M^2 r_M} ,\]
which tends to zero as $M \to \infty$, by~\eqref{eq:exponential-parameter-condition}.
Since $r_M \to \infty$, it follows that
$1 - (r_M-1) (  \psi_M (t/a_M)  - 1) = 1 -  i t + o(1)$, uniformly in $| t | \leq t_0$, and so,
by~\eqref{eq:char-fn}
\[ \lim_{M \to \infty} \Exp [ \re^{i t Z_M/a_M} ] = \frac{1}{1 - i t} , \]
for all $t$ in an open interval containing~$0$, and since $\Exp [ \re^{i t \cExp} ] = (1-it)^{-1}$
is the characteristic function of the unit-mean exponential distribution,
this completes the proof.
\end{proof}

\begin{proof}[Proof of Theorem~\ref{thm:confined-space}.]
To simplify notation, we write
$\theta_M := \theta (N_M, M)$.
First suppose that $X_0 = 1$.
Then in the representation $\sigma_\kappa = \sum_{i=1}^\kappa \nu_i$, 
Lemma~\ref{lem:excursion-law} shows that, given $\kappa = k \in \N$, $\nu_1, \ldots, \nu_k$ are i.i.d.~with the law of $Y_M$ as given by~\eqref{eq:Y-def},
and the law of $\kappa$ is $\PrfMo ( \kappa = k )  = (1-\theta_M)^k \theta_M$, for $k \in \ZP$.
Thus Lemma~\ref{lem:exponential-limit} applies to show that $\sigma_\kappa \to \cExp$ in distribution,
provided that~\eqref{eq:exponential-parameter-condition}
holds, where $p_M = \theta_M$ satisfies~\eqref{eq:feller-tail-supercritical-theta}, and, by~\eqref{eq:weak-conditioning-moments},
$\mu_M = \ExpfMo [ \nu \mid \nu  < \infty ] \sim N_M$ and 
$\sigma^2_M = \VarfMo [ \nu \mid \nu  < \infty ] \sim N_M^3/3$. Hence the quantity in~\eqref{eq:exponential-parameter-condition}
satisfies
\[ \frac{\sigma_M^2 p_M}{\mu_M^2 } = \frac{4}{3} (1+o(1)) \cos^M \left( \frac{\pi}{N_M} \right) \leq\frac{4}{3} (1+o(1)) \exp \left\{ - \frac{\pi^2 M}{2 N_M^2} \right\}
 ,\]
which tends to~$0$ provided that~\eqref{eq:confined-space} holds. Lemma~\ref{lem:exponential-limit}
then establishes~\eqref{eq:lifetime-confined-space} in the case where $\zeta_0 = (1,M)$. 
For general $z_M$ we have (by Lemma~\ref{lem:mean-tau}) that $\ExpfM [ \nu \mid \nu < \infty ] = O ( N_M^2)$,
and, since, by~\eqref{eq:confined-space}, $M /N_M^2 \to \infty$ and $y_M \geq \eps M$
for some $\eps >0$ and all $M$ large enough, 
it follows from the $a=0$, $y = \eps$ case of~\eqref{eq:weak-limit}
that $\theta_{z_M} (N_M, M) \to 0$ as $M \to \infty$.
Hence the first excursion does not change the limit behaviour.
\end{proof}

\subsection{The critical case}
\label{sec:proofs:critical}

Recall the definition of
 $H$ from~\eqref{eq:H-def}.

\begin{proposition}
\label{prop:critical-excursions}
Suppose that~\eqref{eq:critical} holds. 
Then, 
\begin{equation}
\label{eq:critical-theta} 
\theta ( N_M, M) = (4/N_M) (1+o(1)) H (\rho), \text{ as } M \to \infty. \end{equation}
Moreover, for any $s_0 \in (0,\infty)$, as $M \to \infty$, uniformly for $s \in (0,s_0]$,
\begin{equation}
\label{eq:critical-nu-mgf}
\ExpfMo [ \re^{s \nu / N_M^2} \mid \nu < \infty ]  = 1 + \frac{4 s}{N_M} (1+o(1))  \int_0^\rho \re^{s y}   \bigl( H (y) - H(\rho) \bigr) \ud y.
\end{equation}
\end{proposition}
\begin{proof}
We have from Corollary~\ref{cor:feller-tail}\ref{cor:feller-tail-ii} that
$\bPo ( \tau_{0,N}/N^2 > y ) = (4/N) (1+o(1)) H (y)$, as $N \to \infty$,
uniformly in $y \geq y_0 >0$,  where $H$ is defined at~\eqref{eq:H-def}.
In particular, under condition~\eqref{eq:critical}, it follows from continuity of~$H$ that
 $\theta (N_M,M) = \bPo ( \tau_{0,N_M}/N_M^2 > (M +1)/N_M^2 ) = (1+o(1)) \bPo ( \tau_{0,N_M}/N_M^2 > \rho )$ satisfies~\eqref{eq:critical-theta}.

Let $\eps \in (0,\rho)$. Note that, for every $s \in \RP$,
\begin{align}
\label{eq:mgf-split}
 \ExpfMo \left[ \left( \re^{s \nu / N_M^2} -1 \right) \1 { \nu < \infty } \right]
&  = 
\bEo \left[ \left( \re^{s \tau_{0,N_M} / N_M^2} -1 \right) \1 { \eps N_M^2 \leq \tau_{0,N_M} \leq  M+1 } \right]
\nonumber\\ 
& {} \quad{}  +  \bEo \left[ \left( \re^{s \tau_{0,N_M} / N_M^2} -1 \right) \1 { \tau_{0,N_M} < \eps N_M^2 } \right].\end{align}
To estimate the second term on the right-hand side of~\eqref{eq:mgf-split}, we apply~\eqref{eq:mgf-split-by-parts} with~$X = \tau_{0,N_M} / N_M^2$, 
$g(y) = \re^{sy} - 1$,
$a = 0$, and $b = \eps$ to obtain, for every $s \in (0,\infty)$,
\[ \bEo \left[ \left| \re^{s \tau_{0,N_M} / N_M^2} -1 \right| \1 { \tau_{0,N_M} < \eps N_M^2 } \right]
\leq s \int_0^{\eps} \re^{sy} \bPo ( \tau_{0,N_M} > y N_M^2 ) \ud y .\]
Here, it follows from~\eqref{eq:one-sided-approx-2} that there exists $C < \infty$ such that, for all $N \in \N$,
\begin{equation}
\label{eq:uniform-tail-bound}
 \bPo ( \tau_{0,N} \geq y N^2 ) \leq \frac{C}{N} (1+y^{-1/2}) , \text{ for all } y >0.\end{equation}
Thus, for $0 \leq s \leq s_0 < \infty$ and all $\eps \in (0,1)$,
\begin{equation}
\label{eq:mgf-bound-1}
 \bEo \left[ \left| \re^{s \tau_{0,N_M} / N_M^2} -1 \right| \1 { \tau_{0,N_M} < \eps N_M^2 } \right]
\leq \frac{C s_0 \re^{s_0}}{N_M} \int_0^{\eps} (1+ y^{-1/2}) \ud y \leq \frac{C(s_0) \eps^{1/2}}{N_M} ,\end{equation}
where $C(s_0) < \infty$ depends on $s_0$ only. 
To estimate the first term on the right-hand side of~\eqref{eq:mgf-split}, we apply~\eqref{eq:mgf-split-by-parts} with~$X = \tau_{0,N_M} / N_M^2$, 
$g(y) = \re^{sy} - 1$, $a = \eps$, and $b = \rho_M := (M+1)/N_M^2 = \rho + o(1)$ to obtain
\begin{align}
\label{eq:mgf-bound-2}
 & {} \bEo \left[  \left( \re^{s \tau_{0,N_M} / N_M^2} -1 \right) \1 { \eps N_M^2 \leq \tau_{0,N_M} \leq  M+1 }  \right] 
= (\re^{s\eps}-1) \bPo ( \tau_{0,N_M} \geq \eps N_M^2 ) \nonumber\\
& \quad {} - (\re^{s\rho_M}-1) \bPo ( \tau_{0,N_M} > \rho_M N_M^2 )
 + s \int_\eps^{\rho_M} \re^{sy} \bPo ( \tau_{0,N_M} > y N_M^2 ) \ud y ; \end{align}
note that $\rho_M > \eps$ for all~$M$ sufficiently large, since $\eps < \rho$.
Since there is some $C < \infty$ for which $\re^{s\eps}-1 \leq C s_0 \eps$
for all $s \leq s_0$ and all $\eps \in (0,1)$, we have from~\eqref{eq:uniform-tail-bound} that 
the first term on the right-hand side of~\eqref{eq:mgf-bound-2} satisfies
\[ 0 \leq (\re^{s\eps}-1)\bPo ( \tau_{0,N_M} \geq \eps N_M^2 ) \leq \frac{C(s_0) \eps^{1/2}}{N_M}, \]
where, again, $C(s_0) < \infty$ depends on $s_0$ only. Corollary~\ref{cor:feller-tail}\ref{cor:feller-tail-ii} implies that, for any fixed $\eps \in (0,1)$,
\[ \int_\eps^{\rho_M} \re^{sy}  \bPo ( \tau_{0,N_M} > y N_M^2 ) \ud y
= \frac{4}{N_M} (1+o(1)) \int_\eps^\rho \re^{sy} H (y) \ud y ,\]
and
\[  (\re^{s\rho_M}-1) \bPo ( \tau_{0,N_M} > \rho_M N_M^2 ) = \frac{4}{N_M} (1+o(1)) (\re^{s\rho}-1) H(\rho) .\]
Hence, combining~\eqref{eq:mgf-split} with~\eqref{eq:mgf-bound-1} and~\eqref{eq:mgf-bound-2}, and taking $\eps \downarrow 0$, we obtain
\[ \lim_{M \to \infty} \sup_{0<s \leq s_0} \left| \frac{N_M}{4}   \ExpfMo \left[ \left( \re^{s \nu / N_M^2} -1 \right) \1 { \nu < \infty } \right]  - s \int_0^\rho \re^{sy} \bigl( H (y) - H(\rho) \bigr) \ud y \right| = 0 .\]
Finally, we observe that
\[ \ExpfMo [ \re^{s \nu / N_M^2} - 1 \mid \nu < \infty ]  = \frac{ \ExpfMo \left[ \left( \re^{s \nu / N_M^2} -1 \right) \1 { \nu < \infty } \right]}{\PrfMo ( \nu < \infty )} ,\]
and $\PrfMo ( \nu < \infty ) = 1- \theta (N_M,M) \to 1$, by~\eqref{eq:theta-def} and~\eqref{eq:critical-theta}. Now~\eqref{eq:critical-nu-mgf} follows.
\end{proof}

\begin{lemma}
\label{lem:H}
The function  $H$ from~\eqref{eq:H-def} satisfies the following.
\begin{enumerate}[label=(\roman*)]
\item
\label{lem:H-i} As $y \downarrow 0$, $ 2 H(y) \sqrt{2 \pi y} \to 1$.
\item
\label{lem:H-ii}
 As $y \to \infty$, $H(y) = (1+o(1)) \re^{-\pi^2 y /2}$.
\end{enumerate}
\end{lemma}
\begin{proof}
Write $c := \pi^2 /2$, so that~\eqref{eq:H-def} reads $H(y) = \sum_{k=1}^\infty \re^{-c (2k-1)^2 y}$. 
Then, 
\[ 0 \leq H(y) - \re^{-cy} \leq \re^{-9cy} \sum_{k=0}^\infty \re^{-c [ (2k+3)^2 -9 ] y} \leq \re^{-9cy} \sum_{k=0}^\infty \re^{-16c k y} ,\]
and hence $0 \leq H(y) - \re^{-\pi^2 y /2} \leq  \re^{- 4 \pi^2 y}$ for all $y$ sufficiently large.
This yields part~\ref{lem:H-ii}. On the other hand, since $k \mapsto h_k (y)$ is strictly decreasing for fixed $y >0$,
\[ \int_1^\infty \re^{- c (2 t -1)^2 y} \ud t \leq  H (y)  \leq 1 + \int_1^\infty \re^{- c (2 t -1)^2 y} \ud t ,\]
where, by the change of variable $z = (2t-1) \sqrt{2 cy}$,
\[ \int_1^\infty \re^{- c (2 t -1)^2 y} \ud t = \frac{1}{2 \pi \sqrt{y}} \int_{\sqrt{2cy}}^\infty \re^{- z^2 /2} \ud z = \frac{1+o(1)}{2 \sqrt{2\pi y}} ,
\]
as $y \downarrow 0$, which gives part~\ref{lem:H-i}.
\end{proof}

\begin{proof}[Proof of Theorem~\ref{thm:critical}.]
Let $s > 0$, and recall the definition of $\Mgf (s)= \ExpfMo [ \re^{s \nu} \mid \nu < \infty]$ from~\eqref{eq:phi-def}. Then, as at~\eqref{eq:Xi-gf-s}, 
\[ 
\ExpfMo \bigl[ \re^{s \sigma_\kappa / M} \bigr]
= \frac{\theta_M}{1-(1-\theta_M) \psi_M ( s / M ) } , \text{ provided } (1-\theta_M) \psi_M ( s / M ) < 1, \]
where $\theta_M := \theta (N_M, M)$. Combining~\eqref{eq:critical-theta}  and~\eqref{eq:critical-nu-mgf}, 
with the fact that $M \sim \rho N_M^2$, by~\eqref{eq:critical}, we conclude that
\begin{equation}
\label{eq:critical-mgf}  
\lim_{M \to \infty} \ExpfMo \bigl[ \re^{s \sigma_\kappa / M} \bigr]  = \frac{1}{1 - G (\rho, s)},  \text{ provided } G (\rho, s) < 1,
\end{equation} 
where $G$ is defined at~\eqref{eq:G-def}.
Since $\lambda = M +1 + \sigma_\kappa$, by~\eqref{eq:lambda-sigma}, it follows from~\eqref{eq:critical-mgf}  that~$\lambda / M$ converges in distribution to $1 + \xi_\rho$,
where $\Exp [ \re^{s\xi_\rho} ] = \phi_\rho (s)$ as defined at~\eqref{eq:phi-rho-def}. 

From~\eqref{eq:G-def}, we have that, for fixed $\rho >0$, 
$s \mapsto G(\rho, s)$ is differentiable on $s \in (0,\infty)$, 
with derivative $G' ( \rho, s) := \frac{\ud}{\ud s} G (\rho , s)$ satisfying
\[ G' (\rho, 0^+) := \lim_{s \downarrow 0} G' (\rho , s ) = \frac{1}{H ( \rho)} \int_0^1 \bigl( H ( v \rho ) - H(\rho) \bigr) \ud v
= \mu ( \rho) ,\]
as given by~\eqref{eq:mu-def}, as we see from the change of variable $y =  \rho v$.
Hence $\phi_\rho$ is differentiable for $s \in (0,s_\rho)$, with derivative $\phi'_\rho (s) = \frac{G'(\rho,s)}{(1-G(\rho,s))^2}$, and
\[  \phi'_\rho (0^+) := \lim_{s \downarrow 0} \phi'_\rho (s) =  G' (\rho, 0^+) =  \mu(\rho). \]
The convergence of the moment generating function in~\eqref{eq:critical-mgf} in the region $s \in (0,s_0)$
implies convergence also of the mean, $\lim_{M \to \infty} M^{-1} \ExpfMo  \sigma_\kappa 
= \phi_\rho (0^+) = \mu (\rho)$, and hence $\lim_{M \to \infty} M^{-1} \ExpfMo \lambda 
= 1 + \mu (\rho)$. The asymptotics for~$\mu$ stated in~\eqref{eq:mu-asymptotics}  follow from~\eqref{eq:mu-def} with the asymptotics for $H$ in Lemma~\ref{lem:H},
together with the observation that
\[ \int_0^\infty H(y) \ud y = \sum_{k=1}^\infty \int_0^\infty h_k (y)\ud y = \frac{2}{\pi^2} \sum_{k=1}^\infty \frac{1}{(2k-1)^2} = \frac{1}{4},\]
using Fubini's theorem. This completes the proof.
\end{proof}

\section{Concluding remarks}
\label{sec:conclusion}

To conclude, we identify a number of potentially interesting open problems.
\begin{itemize}
\item Staying in the context of the $M$-capacity models, we expect that at least the non-critical
results of the present paper are rather universal, and should extend
to more general domains, and to more general diffusion dynamics within a suitable class. For example,
we expect that for a large class of energy-constrained random walk models in a domain of diameter $N$ with energy capacity $M$,
the regimes $M \ll N^2$ (meagre capacity) and $M \gg N^2$ (confined space) should lead to
Darling--Mandelbrot and exponential limits, respectively, for the total lifetime.
\item The total lifetime is just one statistic associated with the model: other quantities that it would be of interest to study include the location of the
walker on extinction.
\item The model presented in Section~\ref{sec:general-model}
includes several other natural models, in addition to the particular finite-capacity model
with total replenishment that we study here. For example, one could consider a model of infinite capacity
and a random replenishment distribution. In models with unbounded energy, 
there may be positive probability that the walker survives for ever (i.e., transience).
\end{itemize}

\appendix

\section{Darling--Mandelbrot distribution}
\label{sec:kummer}

Recall the definition of the function $K : \R \to \R$ and the set $\cT$ from~\eqref{eq:kummer}.
Write $M(a,b,t)$ for the Kummer function defined for $a , b \notin -\N$ by the convergent series
\[ M(a,b,t) := \frac{\Gamma(b)}{\Gamma(a)} \sum_{\ell=0}^\infty \frac{\Gamma (a+\ell)}{\Gamma (b+\ell)} \frac{t^\ell}{\ell!}, \text{ for } t \in \R ; \]
see~\cite[Chapter~13]{as} or~\cite[pp.~647--8]{bosa}. Then, by~\eqref{eq:kummer}, $K (t) = M (-\frac{1}{2}, \frac{1}{2}, t)$. This identification is the basis for the following
 facts.

\begin{lemma}
\phantomsection
\label{lem:kummer}
  \begin{enumerate}[label=(\alph*)]
\item
\label{lem:kummer-a}
 As $t \to \infty$, we have $K(-t) \sim \sqrt{\pi t}$ and $K(t) \sim - \re^t/(2t)$.
\item 
\label{lem:kummer-b}
The function $K$ is infinitely differentiable, with $K^{(\ell)} (t) := (\ud^\ell / \ud t^\ell ) K (t) < 0$ for every $\ell \in \N$ and all $t \in \R$,
and $K^{(\ell)} (0) = -1/(2\ell -1)$.
\item 
\label{lem:kummer-c}
We have that $\cT = (-\infty, t_0)$ where $t_0$ is uniquely determined by $\sum_{\ell=1}^\infty \frac{t_0^\ell}{(2\ell -1) \cdot \ell!} =1$; numerically,
$t_0 \approx 0.8540326566$.
\item
\label{lem:kummer-d}
With $\cI$ as defined at~\eqref{eq:E-def}, it is the case that
\begin{equation}
\label{eq:K-E}
 K(t) = \re^t - \cI (  t  ) , \text{ for all } t \in \R. 
\end{equation}
\end{enumerate}
\end{lemma}

\begin{remark}
\label{rem:csaki}
The constant $t_0$ in Lemma~\ref{lem:kummer}\ref{lem:kummer-c} also 
appears in relation to asymptotics of~$T_n$, the maximum excursion duration over the first $n$ steps of a simple symmetric random walk 
on $\Z$. It was shown by Cs\'aki, Erd\H os and R\'ev\'esz~\cite[Thm.~1]{cer} that
\[ \liminf_{n \to \infty} \left[ \frac{\log \log n}{n} T_n \right] = t_0, \as  \]
See also~\cite{lh,revesz} for some neighbouring results.
\end{remark}

\begin{proof}[Proof of Lemma~\ref{lem:kummer}.]
The $t \to \pm \infty$ asymptotics can be read off from~(13.1.4) and~(13.1.5) in~\cite{as}.
The derivatives of $K$ are obtained from equation (13.4.9) in~\cite{as} as
\[ \frac{\ud^\ell}{\ud t^\ell} K(t) = \frac{\Gamma \left( \tfrac{1}{2} \right) \Gamma \left( \ell - \tfrac{1}{2} \right) }{\Gamma \left( -\tfrac{1}{2} \right) \Gamma \left( \ell + \tfrac{1}{2} \right) } M \left( \ell - \frac{1}{2} , \ell + \frac{1}{2} , t \right)
= - \frac{1}{2\ell -1 } M \left( \ell - \frac{1}{2} , \ell + \frac{1}{2} , t \right) .\]
Thus $K^{(\ell)} (0) = -1/(2\ell -1)$. 
For $b > a > 0$, 
the integral representation~(13.2.1) in~\cite{as} shows that $M (a,b,t) > 0$ for all $t \in \R$, and hence $(\ud^\ell / \ud t^\ell ) K (t) < 0$ for all $\ell \in \N$, as claimed.
In particular, $K$ is monotone decreasing, with $K(0) = 1$, so there is a unique $t_0 \in \R$ with $K(t_0) = 0$,
one has $t_0 > 0$, 
and $K(t) > 0$ if and only if $t < t_0$. Finally, equations~(13.4.4) and~(13.6.12) in~\cite{as} show that
\[ K (t) = M \left(-\frac{1}{2}, \frac{1}{2} , t \right) = M \left(\frac{1}{2}, \frac{1}{2} , t \right) - 2 t M \left( \frac{1}{2}, \frac{3}{2} , t \right) = \re^t - 2 t M \left( \frac{1}{2}, \frac{3}{2} , t \right) .\]
Moreover, since $\frac{3}{2} > \frac{1}{2} >0$,  equation~(13.2.1) in~\cite{as} shows that
$\cI (t) = 2 t M( \frac{1}{2}, \frac{3}{2} , t )$. This verifies~\eqref{eq:K-E}.
\end{proof}

\begin{lemma}
\label{lem:xi-moments}
The moments $\upsilon_k := \Exp[ \xi^k ]$ $(k\in\N)$ of the random variable~$\xi$ defined by~\eqref{eq:xi-transform} 
are determined uniquely by the recursion
\begin{equation}
\label{eq:xi-moments-recursion}
\upsilon_k = \sum_{j=1}^k \binom{k}{j} \frac{\upsilon_{k-j}}{2j-1} , \text{ for } k \in \N,
\end{equation}
with the initial condition $\upsilon_0 = 1$. In particular, the first few moments are as follows:
\begin{center}
{\normalfont\setlength\extrarowheight{2pt}
\begin{tabular}{c|cccccc}
$k$ & 1 & 2 & 3 & 4 & 5 & 6 \\
\hline
$\upsilon_k$ & $1$ & $\frac{7}{3}$ & $\frac{41}{5}$ & $\frac{4033}{105}$ & $\frac{14167}{63}$ & $\frac{1824719}{1155}$ 
 \end{tabular}}
\end{center}
\end{lemma}
\begin{proof}
From~\eqref{eq:xi-transform} and~\eqref{eq:K-E}, we have that $\phidm (t) = \Exp [ \re^{t\xi} ] = 1 / K(t)$,
where, for $\ell \in \N$, $K^{(\ell)} (0) = - 1/(2\ell -1)$,
as shown in Lemma~\ref{lem:kummer}\ref{lem:kummer-b}. Differentiating $k \in \N$ times the equality $K(t) \phidm (t) = 1$, we obtain
\[ \sum_{j=0}^k \binom{k}{j} K^{(j)} (t) \phidm^{(k-j)} (t) = 0 , \text{ for } k \in \N. \]
Taking $t=0$ and noting that $\nu_k = \phidm^{(k)} (0)$ and $K^{(0)} (0) = K(0) = 1$ gives the result.
\end{proof}

\section*{Acknowledgements}
\addcontentsline{toc}{section}{Acknowledgements}

The authors are grateful for the comments and suggestions of two anonymous referees,
  in particular, their encouragement
to present the heuristics in Section~\ref{sec:heuristics} and their
suggestions for future work, which we incorporated into Section~\ref{sec:conclusion}.
The work of AW was supported by EPSRC grant EP/W00657X/1.

\end{document}